\documentclass[12pt,a4paper]{amsart}

\usepackage[abbrev]{amsrefs}
\usepackage{amssymb}

\theoremstyle{plain}
  \newtheorem{thm}{Theorem}[section]
  \newtheorem{lem}[thm]{Lemma}
  
  \newtheorem{cor}[thm]{Corollary}
  \newtheorem{prop}[thm]{Proposition}

\theoremstyle{definition}
  \newtheorem{defn}[thm]{Definition}

\theoremstyle{remark}

\numberwithin{equation}{section}

\DeclareMathOperator{\ObsDiam}{ObsDiam}
\DeclareMathOperator{\diam}{diam}

\DeclareMathOperator{\supp}{supp}
\DeclareMathOperator{\dP}{\mathit{d}_\mathrm{P}}
\DeclareMathOperator{\dconc}{\mathit{d}_\mathrm{conc}}

\DeclareMathOperator{\dH}{\mathit{d}_\mathrm{H}}
\DeclareMathOperator{\dTV}{\mathit{d}_\mathrm{TV}}

\newcommand{\field}[1]{\mathbb{#1}}

\newcommand{\R}{\field{R}}

\newcommand{\dKF}[1][]{{d_{\rm KF}^{#1}}}

\newcommand{\cX}{\mathcal{X}}
\newcommand{\cL}{\mathcal{L}}

\newcommand{\Lip}{\mathcal{L}\mathit{ip}}
\newcommand{\cP}{\mathcal{P}}

\begin{document}

\title[High-dimensional ellipsoids]
{High-dimensional ellipsoids\\ converge to Gaussian spaces}

\thanks{This work was supported by JSPS KAKENHI Grant Number 19K03459
and 17J02121.}

\begin{abstract}
We prove the convergence of (solid) ellipsoids to a Gaussian space
in Gromov's concentration/weak topology as the dimension diverges to infinity.
This gives the first discovered example of an irreducible nontrivial convergent sequence in the concentration topology,
where `irreducible nontrivial' roughly means to be not constructed from L\'evy families nor box convergent sequences.
\end{abstract}

\author{Daisuke Kazukawa}
\author{Takashi Shioya}

\address{Mathematical Institute, Tohoku University, Sendai 980-8578,
  JAPAN}
\email{shioya@math.tohoku.ac.jp}
\email{daisuke.kazukawa.s6@dc.tohoku.ac.jp}
\date{\today}

\keywords{ellipsoid, concentration topology, Gaussian space, observable distance, box distance, pyramid}

\subjclass[2010]{53C23}

\maketitle

\section{Introduction} \label{sec:intro}

The study of convergence of metric measure spaces is one of
central topics in geometric analysis on metric measure spaces.
We refer to \cites{Gmv:green, St:geomI, St:geomII, LV:Ric, GMS:conv}
for some celebrated works on it.
Such the study originates that of Gromov-Hausdorff convergence/collapsing of
Riemannian manifolds, which has widely been developed and
applied to solutions to many significant problems in geometry and topology,
including Thurston's geometrization conjecture \cites{P:surgery, SyYm:volume}.
As the starting point of geometric analytic study in the collapsing theory,
Fukaya \cite{Fk:Laplace} introduced the concept of measured Gromov-Hausdorff convergence
of metric measure spaces to study the Laplacian of collapsing
Riemannian manifolds.
There, he discovered that only the metric structure but also the measure structure
plays an important role in the collapsing phenomena.
After that, Cheeger-Colding \cites{CC:strI, CC:strII, CC:strIII}
established a theory of
measured Gromov-Hausdorff limits of complete Riemannian
manifolds with a lower bound of Ricci curvature,
which is nowadays widely applied in the Riemannian and K\"ahler geometry.

Meanwhile, Gromov \cite{Gmv:green}*{Chapter 3$\frac{1}{2}_+$}
(see also \cite{Sy:mmg})
has developed a new convergence theory of
metric measure spaces based on the concentration of measure phenomenon
due to L\'evy and V.~Milman \cites{Lv:prob, Ml:heritage, Ld:conc},
where the concentration of measure phenomenon
is roughly stated as that any $1$-Lipschitz function on
high-dimensional spaces is almost constant.
In Gromov's theory, he introduced two fundamental concepts of distance functions,
the \emph{observable distance function} $\dconc$
and the \emph{box distance function} $\square$, on
the set, say $\cX$, of isomorphism classes of metric measure spaces.
The box distance function is nearly
a metrization of measured Gromov-Hausdorff convergence
(precisely the isomorphism classes are little different),
while the observable distance function induces a very characteristic topology,
called the \emph{concentration topology},
which is effective to capture the high-dimensional aspects of spaces.
The concentration topology is weaker than the box topology
and in particular, a measured Gromov-Hausdorff convergence
becomes a convergence in the concentration topology.
He also introduced a natural compactification, say $\Pi$, of $\cX$,
with respect to the concentration topology,
where the topology on $\Pi$ is called the \emph{weak topology}.
The concentration topology is sometimes useful to investigate
the dimension-free properties of manifolds.
For example, it has been applied to obtain a new dimension-free estimate of
eigenvalue ratios of the drifted Laplacian on a closed Riemannian manifold
with nonnegative Bakry-\'Emery Ricci curvature \cite{FS:conc}.

The study of the concentration and weak topologies has been
growing rapidly in recent years (see \cites{FS:conc, Kz:prod, KOS:rough, Nk:isop-error, Oz:quantum, OzSz:Talagrand, OzSy:limf, OzYk:stab, P:inf-gp, Sy:mmg, Sy:mmlim, SyTk:stiefel}).
However, there are only a few nontrivial examples of convergent sequences of
metric measure spaces in the concentration and weak topologies,
where `nontrivial' means neither to be
a L\'evy family (i.e., convergent to a one-point space),
to infinitely dissipate (see Subsection \ref{ssec:dissipation}
for dissipation), nor to be box convergent.
One way to construct a nontrivial convergent sequence
is to take the disjoint union or
the product (more generally the fibration) of trivial sequences and
to perform little surgery on it
(and also to repeat these procedures finitely many times).
We call a sequence obtained in this way a \emph{reducible} sequence.
An \emph{irreducible} sequence is a sequence that is not reducible.
In this paper, any sequence of (solid) ellipsoids
has a subsequence converging to
an infinite-dimensional Gaussian space in the concentration/weak topology.
This provides a new family of nontrivial weak convergent sequences
and especially contains the first discovered example of an irreducible nontrivial sequence that is convergent in the concentration topology.

Let us state our main results precisely.
A solid ellipsoid and an ellipsoid are respectively written as
\begin{align*}
\mathcal{E}^n_{\{\alpha_i\}} &:=
\{\,x \in \R^n \mid \sum_{i=1}^n \frac{x_i^2}{\alpha_i^2} \le 1\,\},\\
\mathcal{S}^{n-1}_{\{\alpha_i\}} &:=
\{\,x \in \R^n \mid \sum_{i=1}^n \frac{x_i^2}{\alpha_i^2} = 1\,\},
\end{align*}
where $\{\alpha_i\}$, $i=1,2,\dots,n$,
is a finite sequence of positive real numbers.
See Section \ref{sec:weak-conv} for the definition of their metric-measure
structures.
Denote by $E^n_{\{\alpha_i\}}$ any one of
$\mathcal{E}^n_{\{\alpha_i\}}$ and $\mathcal{S}^{n-1}_{\{\alpha_i\}}$.
Let us given a sequence $\{E^{n(j)}_{\{\alpha_{ij}\}_i}\}_j$
of (solid) ellipsoids,
where $\{\alpha_{ij}\}$, $i=1,2,\dots,n(j)$, $j=1,2,\dots$,
is a double sequence of positive real numbers.
Our problem is to determine under what condition it will converge
in the concentration/weak topology and to describe its limit.

In the case where the dimension $n(j)$ is bounded for all $j$,
the problem is easy to solve.
In fact, such the sequence has a Hausdorff-convergent subsequence
in a Euclidean space,
which is also box convergent, if $\alpha_{ij}$ is bounded for all $i$ and $j$;
the sequence has an infinitely dissipating subsequence
if $\alpha_{ij}$ is unbounded.

We set $a_{ij} := \alpha_{ij}/\sqrt{n(j)-1}$.
If $n(j)$ and $\sup_i a_{ij}$ both diverge to infinity as $j\to\infty$,
then it is also easy to prove that  $\{E^n_{\{\alpha_i\}}\}$ infinitely dissipates
(see Proposition \ref{prop:aij-diverge}).

By the reason we have mentioned above, we assume
\begin{enumerate}
\item[(A0)] $n(j)$ diverges to infinity as $j\to\infty$
and $a_{ij}$ is bounded for all $i$ and $j$.
\end{enumerate}
We further consider the following three conditions.
\begin{enumerate}
\item[(A1)] $n(j)$ is monotone nondecreasing in $j$.
\item[(A2)] $a_{ij}$ is monotone nonincreasing in $i$ for each $j$.
\item[(A3)] $a_{ij}$ converges to a real number, say $a_i$, as $j\to\infty$
for each $i$.
\end{enumerate}
Note that (A2) and (A3) together imply
that $a_i$ is monotone nonincreasing in $i$.

Any sequence of (solid) ellipsoids with (A0)
contains a subsequence $\{E_j\}$ such that
each $E_j$ is isomorphic to $E^{n(j)}_{\{\sqrt{n(j) - 1}\,a_{ij}\}_i}$
for some sequence $\{a_{ij}\}$ satisfying (A0)--(A3).
In fact, we have a subsequence for which
the dimensions satisfy (A1).
Then, exchanging the axes of coordinate provides (A2).  
A diagonal argument proves to have a subsequence
satisfying (A3).
Thus, our problem becomes to investigate
the convergence of $\{E^{n(j)}_{\{\sqrt{n(j) - 1}\,a_{ij}\}_i}\}_j$
satisfying (A0)--(A3).

One of our main theorems is stated as follows.
Refer to Subsection \ref{ssec:Gaussian} for the definition of
the Gaussian space $\Gamma^\infty_{\{a_i^2\}}$.

\begin{thm} \label{thm:ellipsoid}
Let $\{a_{ij}\}$, $i=1,2,\dots,n(j)$, $j=1,2,\dots$, be a sequence
of positive real numbers satisfying {\rm(A0)--(A3)}.
Then, $E^{n(j)}_{\{\sqrt{n(j) - 1}\,a_{ij}\}_i}$
converges weakly to the infinite-dimensional Gaussian space
$\Gamma^\infty_{\{a_i^2\}}$ as $j \to \infty$.
This convergence becomes a convergence in the concentration topology if and only if $\{a_i\}$ is an $l^2$-sequence.
Moreover, this convergence becomes an asymptotic concentration
{\rm(}i.e., a $\dconc$-Cauchy sequence{\rm)}
if and only if $\{a_i\}$ converges to zero.
\end{thm}

Gromov presents an exercise \cite{Gmv:green}*{3$\frac{1}{2}$.57} 
which is some easier special cases of our theorem.
Our theorem provides not only an answer
but also a complete generalization of his exercise.

For the case of round spheres (i.e., $a_{ij} = a_{1j}$ for all $i$ and $j$)
and also of projective spaces,
the theorem is formerly obtained by the second named author
\cites{Sy:mmg, Sy:mmlim},
for which the convergence is only weak.
Also, the weak convergence of Stiefel and flag manifolds
are studied jointly by Takatsu and the second named author \cite{SyTk:stiefel}.

We emphasize that convergence in the weak/concentration topology is completely different from weak convergence of measures.
For instance, the Prokhorov distance between the normalized volume measure
on $S^{n-1}(\sqrt{n-1})$ and the $n$-dimensional standard Gaussian measure on $\R^n$
is bounded away from zero \cite{SyTk:stiefel},
though they both converge to the infinite-dimensional standard Gaussian space in Gromov's weak topology.

As for the characterization of weak convergence of measures,
we prove in Proposition \ref{prop:box-conv} that,
if $\{a_{ij}\}_i$ $l^2$-converges to an $l^2$-sequence
$\{a_i\}$ as $j\to\infty$,
then the measure of $E^{n(j)}_{\{\sqrt{n(j) - 1}\,a_{ij}\}_i}$
converges weakly to the Gaussian measure $\gamma^\infty_{\{a_i\}}$
on a Hilbert space,
and consequently, the weak convergence
in Theorem \ref{thm:ellipsoid} becomes the box convergence.
Conversely, the $l^2$-convergence of $\{a_{ij}\}_i$
is also a necessary condition for the box convergence of the (solid) ellipsoids
as is seen in the following theorem.

\begin{thm} \label{thm:box-conv}
Let $\{a_{ij}\}$, $i=1,2,\dots,n(j)$, $j=1,2,\dots$, be a sequence
of positive real numbers satisfying {\rm(A0)--(A3)}.
Then, the convergence in Theorem {\rm\ref{thm:ellipsoid}} becomes
a box convergence if and only if we have
\[
\sum_{i=1}^\infty a_i^2 < +\infty
\quad\text{and}\quad
\lim_{j\to\infty} \sum_{i=1}^{n(j)} (a_{ij} - a_i)^2 = 0.
\]
\end{thm}

Theorems \ref{thm:ellipsoid} and \ref{thm:box-conv} together provide
an example of irreducible nontrivial convergent sequence of
metric measure spaces in the concentration topology, i.e., 
the sequence of the (solid) ellipsoids with an $l^2$-sequence $\{a_i\}$
and with a non-$l^2$-convergent $\{a_{ij}\}_i$ as $j\to\infty$.

The proof of the `only if' part of Theorem \ref{thm:box-conv}
is highly nontrivial.
If $\{a_{ij}\}_i$ does not $l^2$-converge, then
it is easy to see that
the measure of the (solid) ellipsoid in such the sequence
does not converge weakly in the Hilbert space.
However, this is not enough to obtain the box non-convergence,
because we consider the \emph{isomorphism classes} of (solid) ellipsoids
for the box convergence.
For the complete proof,
we need a delicate discussion using Theorem \ref{thm:ellipsoid}.

Let us briefly mention the outline of the proof of 
the weak convergence of solid ellipsoids in Theorem \ref{thm:ellipsoid}.
For simplicity, we set $E^n := E^{n(j)}_{\{\sqrt{n(j) - 1}\,a_{ij}\}_i}$ and
$\Gamma := \Gamma^\infty_{\{a_i^2\}}$.
For the weak convergence, it is sufficient to show that
\begin{enumerate}
\item[(1.1)] the limit of $E^n$ dominates $\Gamma$,
\item[(1.2)] $\Gamma$ dominates the limit of $E^n$,
\end{enumerate}
where, for two metric measure spaces $X$ and $Y$,
the space $X$ \emph{dominates} $Y$
if there is a $1$-Lipschitz map from $X$ to $Y$
preserving their measures.

(1.1) easily follows from the Maxwell-Boltzmann distribution law
(Proposition \ref{prop:MB-law}).

(1.2) is much harder to prove.
Let us first consider the simple case where
$E^n$ is the ball $B^n(\sqrt{n-1})$ of radius $\sqrt{n-1}$
and where $\Gamma = \Gamma^\infty_{\{1^2\}}$.
We see that, for any fixed $0 < \theta < 1$,
the $n$-dimensional Gaussian measure $\gamma^n_{\{1^2\}}$ and
the normalized volume measure of $B^n(\theta\sqrt{n-1})$
both are very small for large $n$.
Ignoring this small part $B^n(\theta\sqrt{n-1})$,
we find a measure-preserving isotropic map, say $\varphi$, from $\Gamma^n_{\{1^2\}} \setminus B^n(\theta\sqrt{n-1})$ to the annulus $B^n(\sqrt{n-1}) \setminus B^n(\theta\sqrt{n-1})$,
where we normalize their measures to be probability.
Estimating the Lipschitz constant of $\varphi$,
we obtain (1.2) with error.
This error is estimated and we eventually obtain
the required weak convergence.

We next try to apply this discussion for solid ellipsoids.
We consider the distortion of the above isotropic map $\varphi$
by a linear transformation determined by $\{a_{ij}\}$.
However, the Lipschitz constant of such the distorted isotropic map
is arbitrarily large depending on $\{a_{ij}\}$.
To overcome this problem,
we settle the assumptions (A0)--(A3),
from which the discussion boils down to
the special case where
$a_i = a_N$ for all $i \ge N$ and
$a_{ij} = a_i \ge a_N$ for all $i,j$ and for a (large) number $N$.
In fact, by (A0)--(A3),
the solid ellipsoid $E^n$ for large $n$ and
the Gaussian space $\Gamma$ are both
close to those in the above special case.  
In this special case, the Gaussian measure $\gamma^n_{\{a_i^2\}}$ and
the normalized volume measure of $E^n$ of the domain
\[
\{\,x \in \R^n \setminus \{o\} \mid
\frac{|x_i|}{\|x\|} < \varepsilon \ \text{for any $i = 1,\dots,N-1$}\,\}.
\]
are both almost full for large $n$ and for any fixed $\varepsilon > 0$.
On this domain, we are able to estimate the Lipschitz constant of
the distorted isotropic map.
With some careful error estimates,
letting $\epsilon \to 0+$ and $\theta \to 1-$,
we prove the weak convergence of $E^n$ to $\Gamma$.

\section{Preliminaries}

In this section, we survey the definitions and the facts
needed in this paper.
We refer to \cite{Gmv:green}*{Chapter 3$\frac{1}{2}_+$} and \cite{Sy:mmg}
for more details.

\subsection{Distance between measures}

\begin{defn}[Total variation distance]
  The \emph{total variation distance $\dTV(\mu,\nu)$} of
  two Borel probability measures $\mu$ and $\nu$ on a topological space $X$
  is defined by
  \[
  \dTV(\mu,\nu) := \sup_A |\,\mu(A) - \nu(A)\,|,
  \]
  where $A$ runs over all Borel subsets of $X$.
\end{defn}

If $\mu$ and $\nu$ are both absolutely continuous with respect to
a Borel measure $\omega$ on $X$, then
\[
\dTV(\mu,\nu) = \frac{1}{2} \int_X \left| \frac{d\mu}{d\omega} - \frac{d\nu}{d\omega} \right| \; d\omega,
\]
where $\frac{d\mu}{d\omega}$ is the Radon-Nikodym derivative of $\mu$ with respect to $\omega$.

\begin{defn}[Prokhorov distance]
  The \emph{Prokhorov distance} $\dP(\mu,\nu)$ between two Borel probability
    measures $\mu$ and $\nu$ on a metric space $(X,d_X)$
  is defined to be the infimum of $\varepsilon \geq  0$ satisfying
  \[
  \mu(B_\varepsilon(A)) \ge \nu(A) - \varepsilon
  \]
  for any Borel subset $A \subset X$, where
  $B_\varepsilon(A) := \{\; x \in X \mid d_X(x,A) < \varepsilon\;\}$.
\end{defn}

The Prokhorov metric is a metrization of weak convergence of
Borel probability measures on $X$ provided that $X$ is a separable
metric space.
It is known that $\dP \le \dTV$.

\begin{defn}[Ky Fan distance]
  Let $(X,\mu)$ be a measure space and $Y$ a metric space.
  For two $\mu$-measurable maps $f,g : X \to Y$, we define 
  the \emph{Ky Fan distance $\dKF(f,g)$ between $f$ and $g$}
  to be the infimum of $\varepsilon \ge 0$ satisfying
  \[
  \mu(\{\;x \in X \mid d_Y(f(x),g(x)) > \varepsilon\;\}) \le \varepsilon.
  \]
% We sometimes write $\dKF(f,g)$ by omitting $\mu$.
\end{defn}

$\dKF$ is a pseudo-metric on the set of $\mu$-measurable maps from $X$ to $Y$.
It holds that $\dKF(f,g) = 0$ if and only if $f = g$ $\mu$-a.e.
We have $\dP (f_* \mu, g_* \mu) \leq \dKF(f,g)$,
where $f_*\mu$ is the push-forward of $\mu$ by $f$.

%\begin{lem}\label{lem:dP-me}
%  Let $(X,\mu)$ be a measure space and $Y$ a metric space.
%  For any two $\mu$-measurable maps $f,g : X \to Y$, we have
%  $\dP (f_* \mu, g_* \mu) \leq \dKF(f,g)$.
%\end{lem}

%\begin{prop} \label{prop:dP-dKF}
%  Let $X$ be a topological space with a Borel probability measure $\mu$
%  and $Y$ a metric space.
%  For any two $\mu$-measurable maps $f,g : X \to Y$, we have
%  \[
%  \dP(f_*\mu,g_*\mu) \le \dKF[\mu](f,g).
%  \]
%\end{prop}
Let $p$ be a real number with $p \ge 1$,
and $(X,d_X)$ a complete separable metric space.

\begin{defn} 
The \emph{$p$-Wasserstein distance} between two Borel probability measures
$\mu$ and $\nu$ on $X$ is defined to be
\[
W_p(\mu,\nu) := \inf_{\pi \in \Pi(\mu,\nu)} \left( \int_{X \times X}
d_X(x,x')^p \, d\pi(x,x') \right)^{\frac{1}{p}} (\le +\infty),
\]
where $\Pi(\mu,\nu)$ is the set of couplings between $\mu$ and $\nu$,
i.e., the set of Borel probability measures $\pi$ on $X \times X$
such that $\pi(A \times X) = \mu(A)$ and $\pi(X \times A) = \nu(A)$
for any Borel subset $A \subset X$.
\end{defn}

\begin{lem} \label{lem:Wp}
Let $\mu$ and $\mu_n$, $n=1,2,\dots$, be Borel probability measures
on $X$.
Then the following {\rm(1)} and {\rm(2)} are equivalent to each other.
\begin{enumerate}
\item $W_p(\mu_n,\mu) \to 0$ as $n\to\infty$.
\item $\mu_n$ converges weakly to $\mu$ as $n\to\infty$ and
the $p$-th moment of $\mu_n$ is uniformly bounded:
\[
\limsup_{n\to\infty} \int_X d_X(x_0,x)^p\, d\mu_n(x) < +\infty
\]
for some point $x_0 \in X$.
\end{enumerate}
\end{lem}

It is known that $\dP^2 \le W_1$
and that $W_p \le W_q$ for any $1 \le p \le q$.

%\begin{prop}
%For any Borel probability measures $\mu$ and $\nu$ on $X$,
%we have
%\begin{align*}
%\dP(\mu,\nu)^2 &\le W_1(\mu,\nu), \tag{1}\\
%W_p(\mu,\nu) &\le W_q(\mu,\nu) \quad\text{for $1 \le p \le q$}. \tag{2}
%\end{align*}
%\end{prop}

\subsection{mm-Isomorphism and Lipschitz order}

\begin{defn}[mm-Space]
  Let $(X,d_X)$ be a complete separable metric space
  and $\mu_X$ a Borel probability measure on $X$.
  We call the triple $(X,d_X,\mu_X)$ an \emph{mm-space}.
  We sometimes say that $X$ is an mm-space, in which case
  the metric and the Borel measure of $X$ are respectively indicated by
  $d_X$ and $\mu_X$.
\end{defn}

\begin{defn}[mm-Isomorphism]
  Two mm-spaces $X$ and $Y$ are said to be \emph{mm-isomorphic}
  to each other if there exists an isometry $f : \supp\mu_X \to \supp\mu_Y$
  with $f_*\mu_X = \mu_Y$,
  where $\supp\mu_X$ is the support of $\mu_X$.
  Such an isometry $f$ is called an \emph{mm-isomorphism}.
  Denote by $\cX$ the set of mm-isomorphism classes of mm-spaces.
\end{defn}

Note that $X$ is mm-isomorphic to $(\supp\mu_X,d_X,\mu_X)$.

\emph{We assume that an mm-space $X$ satisfies
\[
X = \supp\mu_X
\]
unless otherwise stated.}

\begin{defn}[Lipschitz order] \label{defn:dom}
  Let $X$ and $Y$ be two mm-spaces.
  We say that $X$ (\emph{Lipschitz}) \emph{dominates} $Y$
  and write $Y \prec X$ if
  there exists a $1$-Lipschitz map $f : X \to Y$ satisfying
  $f_*\mu_X = \mu_Y$.
  We call the relation $\prec$ on $\cX$ the \emph{Lipschitz order}.
\end{defn}

The Lipschitz order $\prec$ is a partial order relation on $\cX$.

\subsection{Observable diameter}

The observable diameter is one of the most fundamental invariants
of an mm-space up to mm-isomorphism.

\begin{defn}[Partial and observable diameter] \label{defn:ObsDiam}
  Let $X$ be an mm-space and let $\kappa > 0$.
  We define
  the \emph{$\kappa$-partial diameter}
  $\diam(X;1-\kappa) = \diam(\mu_X;1-\kappa)$ of $X$
  to be the infimum of the diameter of $A$, 
  where $A \subset X$ runs over all Borel subsets
  with $\mu_X(A) \ge 1-\kappa$.
  Denote by $\Lip_1(X)$ the set of $1$-Lipschitz continuous
  real-valued functions on $X$.
  We define
  the ($\kappa$-)\emph{observable diameter of $X$} by
  \begin{align*}
  \ObsDiam(X;-\kappa) &:= \sup_{f \in \Lip_1(X)} \diam(f_*\mu_X;1-\kappa),\\
  \ObsDiam(X) &:= \inf_{\kappa > 0} \max\{\ObsDiam(X;-\kappa),\kappa\}.
  \end{align*}
\end{defn}

%\begin{defn}[L\'evy family]
%  A sequence of mm-spaces $X_n$, $n=1,2,\dots$,
%  is called a \emph{L\'evy family} if
%  \[
%  \lim_{n\to\infty} \ObsDiam(X_n;-\kappa) = 0
%  \]
%  for any $\kappa > 0$.
%\end{defn}

It is easy to see that
the ($\kappa$-)observable diameter is monotone nondecreasing
with respect to the Lipschitz order relation.

%\begin{prop} \label{prop:ObsDiam-dom}
%  If $X \prec Y$ for two mm-spaces $X$ and $Y$, then
%  \[
%  \ObsDiam(X;-\kappa) \le \ObsDiam(Y;-\kappa)
%  \]
%  for any $\kappa > 0$.
%\end{prop}

%\begin{defn}[Separation distance]
%  Let $X$ be an mm-space.
%  For any real numbers $\kappa_0,\kappa_1,\cdots,\kappa_N > 0$
%  with $N\geq 1$,
%  we define the \emph{separation distance}
%  \[
%  \Sep(X;\kappa_0,\kappa_1, \cdots, \kappa_N)
%  \]
%  of $X$ as the supremum of $\min_{i\neq j} d_X(A_i,A_j)$
%  over all sequences of $N+1$ Borel subsets $A_0,A_1, \cdots, A_N \subset X$
%  satisfying that $\mu_X(A_i) \geq \kappa_i$ for all $i=0,1,\cdots,N$,
%  where $d_X(A_i,A_j) := \inf_{x\in A_i,y\in A_j} d_X(x,y)$.
%  If $\kappa_i > 1$ for some $i$,
%  we define
%  \[
%  \Sep(X;\kappa_0,\kappa_1, \cdots, \kappa_N) := 0.
%  \]
%\end{defn}
%
%Note that if $\sum_{i=0}^N \kappa_i > 1$,
%then at least two of $A_0,\dots,A_N$ has nonempty intersection
%and so
% \[
%\Sep(X;\kappa_0,\kappa_1, \cdots, \kappa_N) = 0.
%\]
%
%\begin{prop} \label{prop:ObsDiam-Sep}
%  Let $X$ be an mm-space.
%  We have, for any $\kappa > \kappa' > 0$,
%  \begin{align}
%    \tag{1} &\ObsDiam(X;-2\kappa) \le \Sep(X;\kappa,\kappa),\\
%    \tag{2} &\Sep(X;\kappa,\kappa) \le \ObsDiam(X;-\kappa').
%  \end{align}
%\end{prop}

\subsection{Box distance and observable distance}

\begin{defn}[Parameter]
  Let $I := [\,0,1\,)$ and let $X$ be an mm-space.
  A map $\varphi : I \to X$ is called a \emph{parameter of $X$}
  if $\varphi$ is a Borel measurable map with $\varphi_*\cL^1 = \mu_X$,
  where $\cL^1$ denotes the one-dimensional Lebesgue measure on $I$.
\end{defn}

It is known that any mm-space has a parameter.

\begin{defn}[Box distance]
  We define the \emph{box distance $\square(X,Y)$ between
    two mm-spaces $X$ and $Y$} to be
  the infimum of $\varepsilon \ge 0$
  satisfying that there exist parameters
  $\varphi : I \to X$, $\psi : I \to Y$, and
  a Borel subset $\tilde{I} \subset I$ such that
  \[
    \cL^1(\tilde{I}) \ge 1-\varepsilon \quad\text{and}\quad
    |\,\varphi^*d_X(s,t)-\psi^*d_Y(s,t)\,| \le \varepsilon
  \]
  for any $s,t \in \tilde{I}$,
  where
  $\varphi^*d_X(s,t) := d_X(\varphi(s),\varphi(t))$ for $s,t \in I$.
\end{defn}

The box metric $\square$ is a complete separable metric on $\cX$.

%\begin{prop} \label{prop:box-dP}
%  Let $X$ be a complete separable metric space.
%  For any two Borel probability measures $\mu$ and $\nu$ on $X$,
%  we have
%  \[
%  \square((X,\mu),(X,\nu)) \le 2 \dP(\mu,\nu).
%  \]
%\end{prop}

\begin{defn}[$\varepsilon$-mm-isomorphism]
Let $\varepsilon$ be a nonnegative real number.
A map $f : X \to Y$ between two mm-spaces $X$ and $Y$
is called an \emph{$\varepsilon$-mm-isomorphism}
if there exists a Borel subset $\tilde{X} \subset X$ such that
\begin{enumerate}
\item[(i)] $\mu_X(\tilde{X}) \ge 1-\varepsilon$,
\item[(ii)] $|\,d_X(x,x')-d_Y(f(x),f(x'))\,| \le \varepsilon$ for any $x,x' \in \tilde{X}$,
\item[(iii)] $\dP(f_*\mu_X,\mu_Y) \le \varepsilon$.
\end{enumerate}
We call the set $\tilde{X}$ a \emph{nonexceptional domain} of $f$.
\end{defn}

%If we have a $0$-mm-isomorphism $f : X \to Y$, then
%there is an mm-isomorphism from $X$ to $Y$
%that coincides with $f$ $\mu_X$-a.e.

%We have the following.

\begin{lem} \label{lem:box-mmiso}
\begin{enumerate}
Let $X$ and $Y$ be two mm-spaces and let $\varepsilon \ge 0$.
\item If there exists an $\varepsilon$-mm-isomorphism from $X$ to $Y$,
then $\square(X,Y) \le 3\varepsilon$.
\item If $\square(X,Y) \le \varepsilon$, then
there exists a $3\varepsilon$-mm-isomorphism from $X$ to $Y$.
\end{enumerate}
\end{lem}

\begin{defn}[Observable distance] \label{defn:dconc}
  For any parameter $\varphi$ of $X$, we set
  \[
  \varphi^*\Lip_1(X)
  := \{\;f\circ\varphi \mid f \in \Lip_1(X)\;\}.
  \]
  We define the \emph{observable distance $\dconc(X,Y)$ between
    two mm-spaces $X$ and $Y$} by
  \[
  \dconc(X,Y) := \inf_{\varphi,\psi} \dH(\varphi^*\Lip_1(X),\psi^*\Lip_1(Y)),
  \]
  where $\varphi : I \to X$ and $\psi : I \to Y$ run over all parameters
  of $X$ and $Y$, respectively,
  and where $\dH$
  is the Hausdorff metric with respect to the Ky Fan metric
  for the one-dimensional Lebesgue measure on $I$.
  $\dconc$ is a metric on $\cX$.
%  We say that a sequence $\{X_n\}$, $n=1,2,\dots$, of mm-spaces
%  \emph{concentrates} to an mm-space $X$
%  if $X_n$ $\dconc$-converges to $X$ as $n\to\infty$.
\end{defn}

%\begin{prop}
%  Let $\{X_n\}_{n=1}^\infty$ be a sequence of mm-spaces.
%  Then, $\{X_n\}$ is a L\'evy family if and only if $X_n$ concentrates to
%  a one-point mm-space as $n\to\infty$.
%\end{prop}

It is known that $\dconc \le \square$ and that
the concentration topology is weaker than the box topology.

\subsection{Pyramid}

\begin{defn}[Pyramid] \label{defn:pyramid}
  A subset $\cP \subset \cX$ is called a \emph{pyramid}
  if it satisfies the following {\rm(i)--(iii)}.
  \begin{enumerate}
  \item[(i)] If $X \in \cP$ and if $Y \prec X$, then $Y \in \cP$.
  \item[(ii)] For any two mm-spaces $X, X' \in \cP$,
    there exists an mm-space $Y \in \cP$ such that
    $X \prec Y$ and $X' \prec Y$.
  \item[(iii)] $\cP$ is nonempty and box closed.
  \end{enumerate}
  We denote the set of pyramids by $\Pi$.
  Note that Gromov's definition of a pyramid is only by (i) and (ii).
  (iii) is added in \cite{Sy:mmg} for the Hausdorff property of $\Pi$.

  For an mm-space $X$ we define
  \[
  \cP X := \{\;X' \in \cX \mid X' \prec X\;\},
  \]
  which is a pyramid.
  We call $\cP X$ the \emph{pyramid associated with $X$}.
\end{defn}

We observe that $X \prec Y$ if and only if $\cP X \subset \cP Y$.
It is trivial that $\cX$ is a pyramid.

We have a metric, denoted by $\rho$, on $\Pi$,
for which we omit to state the definition.
We say that a sequence of pyramids \emph{converges weakly} to
a pyramid if it converges with respect to $\rho$.
We have the following.
\begin{enumerate}
\item The map
$\iota : \cX \ni X \mapsto \cP X \in \Pi$
is a $1$-Lipschitz topological embedding map with respect to
$\dconc$ and $\rho$.
\item $\Pi$ is $\rho$-compact.
\item $\iota(\cX)$ is $\rho$-dense in $\Pi$.
\end{enumerate}
In particular, $(\Pi,\rho)$ is a compactification of $(\cX,\dconc)$.
We say that a sequence of mm-spaces \emph{converges weakly} to a pyramid
if the associated pyramid converges weakly.
Note that we identify $X$ with $\cP X$ in Section \ref{sec:intro}.

For an mm-space $X$, a pyramid $\cP$, and $t > 0$, we define
\[
tX := (X,t\,d_X,\mu_X) \quad\text{and}\quad
t\cP := \{\; tX \mid X \in \cP \;\}.
\]
We see $\cP\,tX = t\,\cP X$.
It is easy to see that $t\cP$ is continuous in $t$
with respect to $\rho$.

We have the following.

\begin{prop} \label{prop:rho-dTV}
  For any two Borel probability measures $\mu$ and $\nu$ on
  a complete separable metric space $X$, we have
  \begin{align*}
  \rho(\cP {(X,\mu)},\cP {(X,\nu)}) &\le \dconc((X,\mu),(X,\nu)) \le \square((X,\mu),(X,\nu)) \\
   &\le 2\dP(\mu,\nu) \le 2\dTV(\mu,\nu).
  \end{align*}
\end{prop}

\subsection{Dissipation} \label{ssec:dissipation}

Dissipation is the opposite notion to concentration.
We omit to state the definition of the infinite dissipation.
Instead, we state the following proposition.
Let $\{X_n\}$, $n=1,2,\dots$, be a sequence of mm-spaces.

\begin{prop}
The sequence $\{X_n\}$ infinitely dissipates
if and only if $\cP X_n$ converges weakly to $\cX$ as $n\to\infty$.
\end{prop}

An easy discussion using \cite{OzSy:limf}*{Lemma 6.6}
leads to the following.

\begin{prop} \label{prop:dissipate}
The following {\rm(1)} and {\rm(2)} are equivalent to each other.
\begin{enumerate}
\item The $\kappa$-observable diameter
$\ObsDiam(X_n;-\kappa)$ diverges to infinity as $n\to\infty$
for any $\kappa \in (\,0,1\,)$.
\item $\{X_n\}$ infinitely dissipates.
\end{enumerate}
\end{prop}

\subsection{Asymptotic concentration}

We say that a sequence of mm-spaces \emph{asymptotically concentrates}
if it is a $\dconc$-Cauchy sequence.
It is known that any asymptotically concentrating sequence converges weakly
to a pyramid.
A pyramid $\cP$ is said to be \emph{concentrated} if $\{(\Lip_1(X)/\sim,\dKF)\}_{X \in \cP}$ is precompact with respect to the Gromov-Hausdorff distance, where $f \sim g$ holds if $f-g$ is constant.

\begin{thm} \label{thm:asymp-conc}
Let $\cP$ be a pyramid.
The following {\rm(1)--(3)} are equivalent to each other.
\begin{enumerate}
\item $\cP$ is concentrated.
\item There exists a sequence of mm-spaces asymptotically concentrating
to $\cP$.
\item If a sequence of mm-spaces converges weakly to $\cP$,
then it asymptotically concentrates.
\end{enumerate}
\end{thm}

\subsection{Gaussian space} \label{ssec:Gaussian}

Let $\{a_i\}$, $i=1,2,\dots,n$, be a finite sequence of nonnegative real numbers.
The product
\[
\gamma^n_{\{a_i^2\}} := \bigotimes_{i=1}^n \gamma^1_{a_i^2}
\]
of the one-dimensional centered Gaussian measure $\gamma^1_{a_i^2}$
of variance $a_i^2$
is an $n$-dimensional centered Gaussian measure on $\R^n$,
where we agree that $\gamma^1_{0^2}$ is the Dirac measure at $0$,
and $\gamma^n_{\{a_i^2\}}$ is possibly degenerate.
We call the mm-space $\Gamma^n_{\{a_i^2\}} := (\R^n,\|\cdot\|,\gamma^n_{\{a_i^2\}})$
the \emph{$n$-dimensional Gaussian space with variance $\{a_i^2\}$}.
Note that, for any Gaussian measure $\gamma$ on $\R^n$,
the mm-space $(\R^n,\|\cdot\|,\gamma)$ is mm-isomorphic to $\Gamma^n_{\{a_i^2\}}$,
where $a_i^2$ are the eigenvalues of the covariance matrix of $\gamma$.

We now take an infinite sequence $\{a_i\}$, $i=1,2,\dots$,
of nonnegative real numbers.
For $1 \le k \le n$, we denote by $\pi^n_k : \R^n \to \R^k$ the natural projection,
i.e.,
\[
\pi^n_k(x_1,x_2,\dots,x_n) := (x_1,x_2,\dots,x_k),
\quad (x_1,x_2,\dots,x_n) \in \R^n.
\]
Since the projection
$\pi^n_{n-1} : \Gamma^n_{\{a_i^2\}} \to \Gamma^{n-1}_{\{a_i^2\}}$
is $1$-Lipschitz continuous and measure-preserving for any $n \ge 2$,
the Gaussian space $\Gamma^n_{\{a_i^2\}}$ is monotone
nondecreasing in $n$ with respect to the Lipschitz order,
so that, as $n\to\infty$, the associated pyramid $\cP {\Gamma^n_{\{a_i^2\}}}$ 
converges weakly to the $\square$-closure of
$\bigcup_{n=1}^\infty \cP {\Gamma^n_{\{a_i^2\}}}$,
denoted by $\cP {\Gamma^\infty_{\{a_i^2\}}}$.
We call $\cP {\Gamma^\infty_{\{a_i^2\}}}$
the \emph{virtual Gaussian space with variance $\{a_i^2\}$}.
We remark that the infinite product measure
\[
\gamma^\infty_{\{a_i^2\}} := \bigotimes_{i=1}^\infty \gamma^1_{a_i^2}
\]
is a Borel probability measure on $\R^\infty$ with respect to the product topology,
but is not necessarily Borel with respect to
the $l^2$-norm.
Only in the case where
\begin{equation}
  \label{eq:sum-finite}
  \sum_{i=1}^\infty a_i^2 < +\infty,
\end{equation}
the measure $\gamma^\infty_{\{a_i^2\}}$ is a Borel measure with respect to
the $l^2$-norm $\|\cdot\|$ which is supported in the separable Hilbert space
$H := \{\;x \in \R^\infty \mid \|x\| < +\infty\;\}$
(cf.~\cite{Bog:Gaussian}*{\S 2.3}), and consequently,
$\Gamma^\infty_{\{a_i^2\}} = (H,\|\cdot\|,\gamma^\infty_{\{a_i^2\}})$
is an mm-space.  In the case of \eqref{eq:sum-finite},
the variance of $\gamma^\infty_{\{a_i^2\}}$ satisfies
\[
\int_{\R^n} \|x\|^2 \, d\gamma^\infty_{\{a_i^2\}}(x) = \sum_{i=1}^\infty a_i^2.
\]

\section{Weak convergence of ellipsoids} \label{sec:weak-conv}

In this section we prove Theorem \ref{thm:ellipsoid}.
We also prove the convergence of Gaussian spaces
as a corollary to the theorem.

Let $\{\alpha_i\}$, $i=1,2,\dots,n$, be a sequence of positive real numbers.
The $n$-dimensional solid ellipsoid $\mathcal{E}^n$
and the $(n-1)$-dimensional ellipsoid $\mathcal{S}^{n-1}$
(defined in Section \ref{sec:intro})
are respectively obtained as the image of 
the closed unit ball $B^n(1)$ and the unit sphere $S^{n - 1}(1)$ in $\R^n$
by the linear isomorphism
$L^n_{\{\alpha_i\}} : \R^n \to \R^n$ defined by
\[
L^n_{\{\alpha_i\}}(x) := (\alpha_1 x_1,\dots,\alpha_n x_n),
\quad x = (x_1,\dots,x_n) \in \R^n.
\]
We assume that the $n$-dimensional solid ellipsoid $\mathcal{E}^n_{\{\alpha_i\}}$
is equipped with the restriction of the Euclidean distance function and
with the normalized Lebesgue measure $\epsilon^n_{\{\alpha_i\}}
:= \widetilde{\mathcal{L}^n|_{\mathcal{E}^n_{\{\alpha_i\}}}}$,
where $\widetilde{\mu} := \mu(X)^{-1} \mu$
is the \emph{normalization} of a finite measure $\mu$ on a space $X$
and $\mathcal{L}^n$ the $n$-dimensional Lebesgue measure
on $\R^n$.
The $(n - 1)$-dimensional ellipsoid $\mathcal{S}^{n - 1}_{\{\alpha_i\}}$
is assumed to be equipped with the restriction of the Euclidean distance function and
with the push-forward
$\sigma^{n - 1}_{\{\alpha_i\}} := (L^n_{\{\alpha_i\}})_*\sigma^{n - 1}$
of the normalized volume measure $\sigma^{n - 1}$
on the unit sphere $S^{n - 1}(1)$ in $\R^n$.

\emph{Throughout this paper,
let $(E_{\{\alpha_i\}}^n,e^n_{\{\alpha_i\}})$ be any one of
\[
(\mathcal{E}_{\{\alpha_i\}}^n,\epsilon^n_{\{\alpha_i\}}) \quad\text{and}\quad
(\mathcal{S}_{\{\alpha_i\}}^{n-1},\sigma^{n-1}_{\{\alpha_i\}})
\]
for any $n \ge 2$ and $\{\alpha_i\}$.}
The measure $e^n_{\{\alpha_i\}}$ is sometimes
considered as a Borel measure on $\R^n$,
supported on $E_{\{\alpha_i\}}^n$.

\begin{lem} \label{lem:domE}
Let $\{\alpha_i\}$ and $\{\beta_i\}$, $i=1,2,\dots,n$, be two sequences
of positive real numbers.
If $\alpha_i \le \beta_i$ for all $i=1,2,\dots,n$, then
$E_{\{\alpha_i\}}^n$ is dominated by $E_{\{\beta_i\}}^n$.
\end{lem}

\begin{proof}
The map $L_{\{\alpha_i/\beta_i\}}^n : E_{\{\beta_i\}}^n \to E_{\{\alpha_i\}}^n$
is $1$-Lipschitz continuous and preserves their measures.
\end{proof}

\begin{prop}[Maxwell-Boltzmann distribution law] \label{prop:MB-law}
Let $\{a_{ij}\}$, $i=1,2,\dots,n(j)$, $j=1,2,\dots$,
be a sequence of positive real numbers satisfying {\rm(A0)} and {\rm(A3)}.
Then, $(\pi^{n(j)}_k)_*e^n_{\{\sqrt{n(j)-1}\,a_{ij}\}_i}$ converges weakly to
$\gamma^k_{\{a_i\}}$ as $j\to\infty$ for any fixed positive integer $k$,
where $\pi^n_k$ is defined in Subsection {\rm\ref{ssec:Gaussian}}.
\end{prop}

\begin{proof}
The proposition follows from a straightforward and standard
calculation (see \cite{Sy:mmg}*{Proposition 2.1}).
\end{proof}

\begin{prop} \label{prop:aij-diverge}
Let $\{a_{ij}\}$, $i=1,2,\dots,n(j)$, $j=1,2,\dots$,
be a sequence of positive real numbers.
If $\sup_i a_{ij}$ diverges to infinity as $j\to\infty$,
then $\{E^{n(j)}_{\{\sqrt{n(j) - 1} \, a_{ij}\}_i}\}$ infinitely dissipates.
\end{prop}

\begin{proof}
Assume that $\sup_i a_{ij}$ diverges to infinity as $j\to\infty$.
Exchanging the coordinates, we assume that
$a_{1j}$ diverges to infinity as $j\to\infty$.
We take any positive real number $a$ and fix it.
Let $\hat{a}_{ij} := \min\{a_{ij},a\}$.
Note that $\hat{a}_{1j} = a$ for all sufficiently large $j$.
By Lemma \ref{lem:domE}, the $1$-Lipschitz continuity of $\pi^{n(j)}_1$,
and the Maxwell-Boltzmann distribution law (Proposition \ref{prop:MB-law}),
we have
\begin{align*}
&\liminf_{j\to\infty} \ObsDiam(E^{n(j)}_{\{\sqrt{n(j) - 1} \, a_{ij}\}_i};-\kappa)\\
&\ge \liminf_{j\to\infty} \ObsDiam(E^{n(j)}_{\{\sqrt{n(j) - 1} \, \hat{a}_{ij}\}_i};-\kappa) \\
&\ge \lim_{j\to\infty} \diam((\pi^{n(j)}_1)_* e^{n(j)}_{\{\sqrt{n(j) - 1} \, \hat{a}_{ij}\}_i};1-\kappa),\\
&= \diam(\gamma^1_a;1-\kappa),
\end{align*}
which diverges to infinity as $a\to\infty$.
Proposition \ref{prop:dissipate} leads us to the dissipation property
for $\{E^{n(j)}_{\{\sqrt{n(j) - 1} \, a_{ij}\}_i}\}$.
\end{proof}

Let $\{a_i\}$, $i=1,2,\dots,n$, be a sequence of positive real numbers.
Let us construct a transport map
from $\gamma^n_{\{a_i^2\}}$ to $\epsilon^n_{\{\sqrt{n-1}\,a_i\}}$.
For $r \ge 0$ we determine a real number $R = R(r)$
in such a way that
$0 \le R \le \sqrt{n-1}$
and
$\gamma^n_{\{1^2\}}(B_r(o)) = \epsilon^n_{\sqrt{n-1}}(B_R(o))$,
where $\epsilon^n_{\sqrt{n-1}}$ denotes the normalized Lebesgue measure on
$\mathcal{E}^n_{\sqrt{n-1}} := B_{\sqrt{n-1}}(o) \subset \R^n$.
Define an isotropic map $\bar{\varphi} : \R^n \to \mathcal{E}^n_{\sqrt{n-1}}$ by
\[
\bar{\varphi}(x) := \frac{R(\|x\|)}{\|x\|} x, \qquad x \in \R^n.
\]
We remark that $\bar{\varphi}_*\gamma^n_{\{1^2\}} = \epsilon^n_{\sqrt{n-1}}$.
It holds that
\[
R = (n-1)^{\frac{1}{2}} \left( \frac{1}{I_{n-1}} \int_0^r t^{n-1} e^{-\frac{t^2}{2}} \, dt \right)^{\frac{1}{n}},
\]
where
\[
I_m := \int_0^\infty t^m e^{-\frac{t^2}{2}} \, dt.
\]
Note that $R$ is strictly monotone increasing in $r$.
Let $L := L^n_{\{a_i\}}$ and $r := r(x) := \|L^{-1}(x)\|$.
We define
\[
\varphi^{\mathcal{E}} := L \circ \bar{\varphi} \circ L^{-1} : \R^n \to \mathcal{E}^n_{\{\sqrt{n-1}\,a_i\}}.
\]
The map $\varphi^{\mathcal{E}}$ is a transport map
from $\gamma^n_{\{a_i^2\}}$ to $\epsilon^n_{\{\sqrt{n-1}\,a_i\}}$,
i.e.,
$\varphi^{\mathcal{E}}_*\gamma^n_{\{a_i^2\}} = \epsilon^n_{\{\sqrt{n-1}\,a_i\}}$.
It holds that $\varphi^{\mathcal{E}}(x) = \frac{R}{r} x$ if $x \neq o$.
We denote by
$\varphi^{\mathcal{S}} : \R^n \setminus \{o\} \to \mathcal{S}^{n - 1}_{\{\sqrt{n-1}\,a_i\}}$
the central projection with center $o$, i.e.,
\[
\varphi^{\mathcal{S}}(x) = \frac{\sqrt{n-1}}{r} \, x,
\qquad x \in \R^n \setminus \{o\},
\]
which is a transport map from $\gamma^n_{\{a_i^2\}}$
to $\sigma^{n - 1}_{\{\sqrt{n-1}\,a_i\}}$.

For an integer $N$ with $1 \le N \le n$
and for $\varepsilon > 0$, we define 
\[
D^n_{N,\varepsilon} := \{\,x \in \R^n \setminus \{o\} \mid
\frac{|x_j|}{\|x\|} < \varepsilon \ \text{for any $j = 1,\dots,N-1$}\,\}.
\]
For $0 < \theta < 1$, let
\[
F^n_\theta := \{\,x \in \R^n \mid \|L^{-1}(x)\| \ge \theta\sqrt{n}\,\}.
\]

\begin{lem} \label{lem:Lip}
We assume that
  \begin{enumerate}
  \item[(i)] $a_i \ge a$ for any $i = 1,2,\dots,n$,
  \item[(ii)] $a_i = a$ for any $i$ with $N \le i \le n$
  and for a positive integer $N$ with $N \le n$.
  \end{enumerate}
Then, there exists a universal positive real number $C$ 
such that,
for any two real numbers $\theta$ and $\varepsilon$ with $0 < \theta < 1$
and $0 < \varepsilon \le 1/N$,
the operator norms of the differentials of $\varphi^{\mathcal{E}}$
and $\varphi^{\mathcal{S}}$ satisfy
\[
\|d\varphi^{\mathcal{E}}_x\| \le \frac{\sqrt{1+CN\varepsilon}}{\theta}
\quad\text{and}\quad
\|d\varphi^{\mathcal{S}}_x\| \le \frac{\sqrt{1+CN\varepsilon}}{\theta}
\]
for any $x \in D^n_{N,\varepsilon} \cap F^n_\theta$.
\end{lem}

\begin{proof}
Let $x \in D^n_{N,\varepsilon} \cap F^n_\theta$ be any point.
We first estimate $\|d\varphi^{\mathcal{E}}_x\|$.
Take any unit vector $v \in \R^n$.
We see that
\begin{align*}
\|d\varphi^{\mathcal{E}}_x(v)\|^2
&= \sum_{j=1}^n \left( \frac{\partial}{\partial r}\left(\frac{R}{r}\right)
\frac{\partial r}{\partial x_j} \langle x,v\rangle
+ \frac{R}{r} v_j \right)^2\\
&= \frac{1}{r^2} \left(\frac{\partial}{\partial r}\left(\frac{R}{r}\right)\right)^2
\langle x,v\rangle^2 \sum_{j=1}^n \frac{x_j^2}{a_j^4}\\
&\quad + 2 \frac{R}{r^2} \frac{\partial}{\partial  r}\left(\frac{R}{r}\right)
\langle x,v\rangle \sum_{j=1}^n \frac{v_j x_j}{a_j^2} + \frac{R^2}{r^2}.
\end{align*}
It follows from (i), (ii), and $x \in D^n_{N,\varepsilon}$ that
\begin{align*}
\frac{a^2 r^2}{\|x\|^2}
&= 1 + \sum_{j=1}^{N-1} \left(\frac{a^2}{a_j^2} - 1\right) \frac{x_j^2}{\|x\|^2}
= 1+O(N\varepsilon^2)\\
\intertext{and so}
\frac{ar}{\|x\|} &= 1+O(N\varepsilon^2), \qquad
\frac{\|x\|}{ar} = 1+O(N\varepsilon^2).
\end{align*}
We also have
\begin{align*}
\frac{a^4}{\|x\|^2} \sum_{j=1}^n \frac{x_j^2}{a_j^4}
&= 1 + \sum_{j=1}^{N-1} \left(\frac{a^4}{a_j^4} - 1\right) \frac{x_j^2}{\|x\|^2}
= 1+O(N\varepsilon^2),\\
\frac{a^2}{\|x\|} \sum_{j=1}^n \frac{v_j x_j}{a_j^2}
&= \sum_{j=1}^n \frac{v_j x_j}{\|x\|}
+ \sum_{j=1}^{N-1} \left(\frac{a^2}{a_j^2} - 1\right) \frac{v_j x_j}{\|x\|}
= \frac{\langle x,v\rangle}{\|x\|} + O(N\varepsilon)
\end{align*}
By these formulas, setting $t := \langle x, v\rangle / \|x\|$
and $g := r \frac{\partial}{\partial r} \left(\frac{R}{r}\right)$,
we  have
\begin{align} \label{eq:dphi}
\|d\varphi^{\mathcal{E}}_x(v)\|^2
&= t^2 g^2 (1+O(N\varepsilon^2))
+ \frac{2t^2 Rg}{r} (1+O(N\varepsilon^2))\\
&\quad + \frac{2tRg}{r} O(N\varepsilon) + \frac{R^2}{r^2}. \notag
\end{align}
We are going to estimate $g$.
Letting $f(r) := \int_0^r t^{n-1} e^{-\frac{t^2}{2}} \, dt$,
we have
\[
\frac{\partial R}{\partial r}
= \sqrt{n-1} \, n^{-1} I_{n-1}^{-\frac{1}{n}} f(r)^{\frac{1}{n}-1} r^{n-1} e^{-\frac{r^2}{2}}
\le n^{-\frac{1}{2}} f(r)^{-1} r^{n-1} e^{-\frac{r^2}{2}},
\]
which together with $f(r) \ge \int_0^r t^{n-1}\,dt = \frac{r^n}{n}$
and $r \ge \theta\sqrt{n}$
yields
\[
0 \le \frac{\partial R}{\partial r} \le \frac{\sqrt{n}}{r}
\le \frac{1}{\theta}.
\]
Since $R \le \sqrt{n-1}$ and $r \ge \theta\sqrt{n}$, we have
$0 \le R/r < 1/\theta$.
Therefore,
\[
|g| = \left| \frac{\partial R}{\partial r} - \frac{R}{r} \right|
\le \frac{1}{\theta}.
\]
Thus, \eqref{eq:dphi} is reduced to
\begin{align*}
\|d\varphi^{\mathcal{E}}_x(v)\|^2
&= t^2 g^2 + \frac{2t^2Rg}{r} + \frac{R^2}{r^2}
+ O(\theta^{-2}N\varepsilon)\\
&= t^2 \left(\frac{\partial R}{\partial r}\right)^2
+ (1-t^2) \frac{R^2}{r^2}
+ O(\theta^{-2}N\varepsilon)\\
&\le \theta^{-2} + O(\theta^{-2}N\varepsilon).
\end{align*}
This completes the required estimate of $\|d\varphi^{\mathcal{E}}_x(v)\|$.

If we replace $R$ with $\sqrt{n-1}$,
then $\varphi^{\mathcal{E}}$ becomes $\varphi^{\mathcal{S}}$
and the above formulas are all true also for $\varphi^{\mathcal{S}}$.
This completes the proof.
\end{proof}

We now give an infinite sequence $\{a_i\}$, $i=1,2,\dots$, of positive real numbers and a positive real number $a$.
Consider the following two conditions.
\begin{enumerate}
\item[(a1)] $a_i \ge a$ for any $i$.
\item[(a2)] $a_i = a$ for any $i \ge N$ and for a positive integer $N$.
\end{enumerate}

\begin{lem} \label{lem:DF}
If we assume {\rm(a2)},
then, for any real numbers $0 < \theta < 1$ and $\varepsilon > 0$, we have
\begin{align}
\lim_{n\to\infty} e^n_{\{\sqrt{n-1}\,a_i\}}(D^n_{N,\varepsilon})
&= 1, \tag{1}\\
\lim_{n\to\infty} \gamma^n_{\{a_i^2\}}(D^n_{N,\varepsilon} \cap F^n_\theta)
&= 1. \tag{2}
\end{align}
\end{lem}

\begin{proof}
Lemma \cite{Sy:mmg}*{Lemma 7.41} tells us
that $\gamma^n_{\{a_i^2\}}(F_\theta^n) = \gamma^n_{\{1^2\}}(L^{-1}(F_\theta^n))$
tends to $1$ as $n\to\infty$.
It holds that
$e^n_{\{\sqrt{n-1}\, a_i\}}(D_{N,\varepsilon}^n)
= \sigma_{\{\sqrt{n-1}\,a_i\}}^{n-1}(D_{N,\varepsilon}^n)
= \gamma^n_{\{a_i^2\}}(D_{N,\varepsilon}^n)$.
The Maxwell-Boltzmann distribution law leads us that
$\sigma_{\{\sqrt{n-1}\,a_i\}}^{n-1}(D_{N,\varepsilon}^n)$
converges to $1$ as $n\to\infty$.
This completes the proof.
\end{proof}

\begin{lem} \label{lem:1}
Assume {\rm(a1)} and {\rm(a2)}.
If a subsequence of $\{\cP {E^n_{\{\sqrt{n - 1} \, a_i\}}}\}_n$
converges weakly to a pyramid $\cP_\infty$ as $n\to\infty$,
then
\[
\cP_\infty \subset \cP {\Gamma^\infty_{\{a_i^2\}}}.
\]
\end{lem}

\begin{proof}
Take any real number $\varepsilon$ with $0 < \varepsilon < 1/N$
and fix it.  Let $\theta := 1/\sqrt{1+CN\varepsilon}$,
where $C$ is the constant in Lemma \ref{lem:Lip}.
Note that $\theta$ satisfies $0 < \theta < 1$
and tends to $1$ as $\varepsilon \to 0+$.
We apply Lemma \ref{lem:Lip}.
Let
\[
\varphi :=
\begin{cases}
\varphi^{\mathcal{E}} &\text{if $(E^n_{\{\sqrt{n-1}\,a_i\}},e^n_{\{\sqrt{n-1}\,a_i\}}) = (\mathcal{E}^n_{\{\sqrt{n-1}\,a_i\}},\epsilon^n_{\{\sqrt{n-1}\,a_i\}})$,}\\
\varphi^{\mathcal{S}} &\text{if $(E^n_{\{\sqrt{n-1}\,a_i\}},e^n_{\{\sqrt{n-1}\,a_i\}}) = (\mathcal{S}^{n-1}_{\{\sqrt{n-1}\,a_i\}},\sigma^{n-1}_{\{\sqrt{n-1}\,a_i\}})$}.
\end{cases}
\]
Since $\varphi$ is $\theta^{-2}$-Lipschitz continuous on $D^n_{N,\varepsilon} \cap F^n_\theta$
and since
$\varphi_*(\widetilde{\gamma^n_{\{a_i^2\}}|_{D^n_{N,\varepsilon} \cap F^n_\theta}})
= \widetilde{e^n_{\{\sqrt{n-1}\,a_i\}}|_{D^n_{N,\varepsilon}}}$,
the $\theta^2$-scale change $\theta^2 X_n$ of the mm-space
$X_n := (\R^n,\|\cdot\|,\widetilde{e^n_{\{\sqrt{n-1}\,a_i\}}|_{D^n_{N,\varepsilon}}})$
is dominated by $Y_n := (\R^n,\|\cdot\|,\widetilde{\gamma^n_{\{a_i^2\}}|_{D^n_{N,\varepsilon} \cap F^n_\theta}})$ and so
$\theta^2 \cP {X_n} = \cP {\theta^2 X_n} \subset \cP {Y_n}$ for any $n$.
Combining Lemma \ref{lem:DF} with Proposition \ref{prop:rho-dTV},
we see that, as $n\to\infty$,
\begin{align*}
\rho(\theta^2\cP {X_n},\cP {E^n_{\{\sqrt{n-1}\,a_i\}}})
&\le 2\dTV(\widetilde{e^n_{\{\sqrt{n-1}\,a_i\}}|_{D^n_{N,\varepsilon}}},e^n_{\{\sqrt{n-1}\,a_i\}}) \to 0,\\
\rho(\cP {Y_n},\cP {\Gamma^n_{\{a_i^2\}}})
&\le 2\dTV(\widetilde{\gamma^n_{\{a_i^2\}}|_{D^n_{N,\varepsilon} \cap F^n_\theta}},
\gamma^n_{\{a_i^2\}}) \to 0.
\end{align*}
Therefore, $\theta^2\cP_\infty$ is contained in $\cP {\Gamma^\infty_{\{a_i^2\}}}$. As $\varepsilon \to 0+$, we have $\theta \to 1$ and
$\theta^2\cP_\infty \to \cP_\infty$.
This completes the proof.
\end{proof}

\begin{lem} \label{lem:2}
  If we assume {\rm(a2)},
  then $\cP {E^n_{\{\sqrt{n - 1} \, a_i\}}}$ converges weakly to
  $\cP {\Gamma^\infty_{\{a_i^2\}}}$
  as $n\to\infty$.
\end{lem}

\begin{proof}
Assume (a2) and
suppose that $\cP {E^n_{\{\sqrt{n - 1} \, a_i\}}}$ does not converge weakly to
$\cP {\Gamma^\infty_{\{a_i^2\}}}$ as $n\to\infty$.
Then, there is a subsequence $\{n(j)\}$ of $\{n\}$ such that
$\cP E^{n(j)}_{\{\sqrt{n(j) - 1} \, a_i\}}$ converges weakly to a pyramid $\cP_\infty$ different from
$\cP {\Gamma^\infty_{\{a_i^2\}}}$.

The Maxwell-Boltzmann distribution law tells us that
the push-forward measure
$\nu^k_{n(j)} := (\pi^{n(j)}_k)_* e^{n(j)}_{\{\sqrt{n(j) - 1}\,a_i\}}$
converges weakly to $\gamma^k_{\{a_i\}}$
as $j\to\infty$ for any $k$, so that
$(\R^k,\|\cdot\|,\nu^k_{n(j)})$
box converges to $\Gamma^k_{\{a_i^2\}}$.
Since $E^{n(j)}_{\{\sqrt{n(j) - 1} \, a_i\}}$ dominates $(\R^k,\|\cdot\|,\nu^k_{n(j)})$,
the limit pyramid $\cP_\infty$ contains $\Gamma^k_{\{a_i^2\}}$ for any $k$.
This proves
\begin{equation} \label{eq:lem2sub}
  \cP_\infty \supset \cP {\Gamma^\infty_{\{a_i^2\}}}.
\end{equation}

Let $\hat{a}_i := \max\{a_i,a\}$.
It follows from $a_i \le \hat{a}_i$ that
$E^n_{\{\sqrt{n - 1} \, a_i\}}$ is dominated by
$E^n_{\{\sqrt{n - 1} \, \hat{a}_i\}}$,
which implies
$\cP {E^n_{\{\sqrt{n - 1} \, a_i\}}} \subset \cP {E^n_{\{\sqrt{n - 1} \, \hat{a}_i\}}}$
for any $n$.
By applying Lemma \ref{lem:1},
the limit of any weakly convergent sequence of $\{\cP {E^n_{\{\sqrt{n - 1} \, \hat{a}_i\}}}\}_n$
is contained in $\cP {\Gamma^\infty_{\{\hat{a}_i^2\}}}$.
Therefore, $\cP_\infty$ is contained in $\cP {\Gamma^\infty_{\{\hat{a}_i^2\}}}$.
Denote by $l$ the number of $i$'s with $a_i < a$.
For any $k \ge N$, we consider the projection from $\Gamma^{k+l}_{\{a_i^2\}}$ to
$\Gamma^k_{\{\hat{a}_i^2\}}$ dropping the axes $x_i$ with $a_i < a$,
which is $1$-Lipschitz continuous and preserves their measures.
This shows that $\Gamma^{k+l}_{\{a_i^2\}}$ dominates
$\Gamma^k_{\{\hat{a}_i^2\}}$,
and so
$\cP {\Gamma^\infty_{\{a_i^2\}}} \supset \cP {\Gamma^\infty_{\{\hat{a}_i^2\}}}$.
We thus obtain
\begin{equation} \label{eq:lem2sup}
  \cP_\infty \subset \cP {\Gamma^\infty_{\{a_i^2\}}}.
\end{equation}

Combining \eqref{eq:lem2sub} and \eqref{eq:lem2sup} yields
$\cP_\infty = \cP {\Gamma^\infty_{\{a_i^2\}}}$,
which is a contradiction.
This completes the proof.
\end{proof}

\begin{lem} \label{lem:supset}
Let $\{a_{ij}\}$ satisfy {\rm(A0)--(A3)}.
If $\cP {E^{n(j)}_{\{\sqrt{n(j) - 1} \, a_{ij}\}_i}}$ converges weakly to a pyramid $\cP_\infty$ as $j\to\infty$,
then
\[
\cP_\infty \supset \cP {\Gamma^\infty_{\{a_i^2\}}}.
\]
\end{lem}

\begin{proof}
Note that the sequence $\{a_i\}$ is monotone nonincreasing.
Put $i_0 := \sup\{\;i \mid a_i > 0\;\}$ ($\le \infty$).
We see $a_{i_0} > 0$ if $i_0 < \infty$.
The Maxwell-Boltzmann distribution law proves that
$\nu^k_{n(j)} := (\pi^{n(j)}_k)_* e^{n(j)}_{\{\sqrt{n(j) - 1}\,a_{ij}\}_i}$ converges weakly to $\gamma^k_{\{a_i^2\}}$
as $j\to\infty$ for each finite $k$ with $1 \le k \le i_0$.
The ellipsoid $E^{n(j)}_{\{\sqrt{n(j) - 1} \, a_{ij}\}_i}$ dominates
$(\R^k,\|\cdot\|,\nu^k_{n(j)})$,
which converges to $\Gamma^k_{\{a_i^2\}}$,
so that $\Gamma^k_{\{a_i^2\}}$ belongs to $\cP_\infty$.
Since
$\Gamma^k_{\{a_i^2\}}$ for any $k \ge i_0$ is mm-isomorphic to $\Gamma^{i_0}_{\{a_i^2\}}$
provided $i_0 < \infty$, we obtain the lemma.
\end{proof}

\begin{lem} \label{lem:subset}
Let $\{a_{ij}\}$ satisfy {\rm(A0)--(A3)}.
If $\cP {E^{n(j) - 1}_{\{\sqrt{n(j) - 1} \, a_{ij}\}_i}}$ converges weakly to a pyramid $\cP_\infty$ as $j\to\infty$,
then
\[
\cP_\infty \subset \cP {\Gamma^\infty_{\{a_i^2\}}}.
\]
\end{lem}

\begin{proof}
Since $\{a_i\}$ is monotone nonincreasing, it converges to a nonnegative real number, say $a_\infty$.

We first assume that $a_\infty > 0$.
We see that $a_i > 0$ for any $i$.
For any $\varepsilon > 0$ there is a number $I(\varepsilon)$
such that 
\begin{equation} \label{eq:ai}
  a_i \le (1+\varepsilon) a_\infty  \qquad\text{for any $i \ge I(\varepsilon)$.}
\end{equation}
Also, there is a number $J(\varepsilon)$ such that
\begin{equation} \label{eq:aij}
 a_{ij} \le a_i + a_\infty\varepsilon \qquad\text{for any $i \le I(\varepsilon)$ and $j \ge J(\varepsilon)$.}
\end{equation}
By the monotonicity of $a_{ij}$ in $i$, \eqref{eq:ai}, and \eqref{eq:aij},
we have
\begin{equation} \label{eq:aijb1}
a_{ij} \le a_{I(\varepsilon),j} \le a_{I(\varepsilon)} + a_\infty\varepsilon \le (1+2\varepsilon)a_\infty
\ \text{for any $i \ge I(\varepsilon)$ and $j \ge J(\varepsilon)$.}
\end{equation}
It follows from \eqref{eq:aij} and $a_\infty \le a_i$ that
\begin{equation} \label{eq:aijb2}
a_{ij} \le a_i + a_\infty\varepsilon \le (1+\varepsilon)a_i
\quad\text{for any $i \le I(\varepsilon)$ and $j \ge J(\varepsilon)$.}
\end{equation}
Let
\[
b_{\varepsilon,i} :=
\begin{cases}
  a_i &\text{if $i \le I(\varepsilon)$,}\\
  a_\infty &\text{if $i > I(\varepsilon)$.}
\end{cases}
\]
By \eqref{eq:aijb1} and \eqref{eq:aijb2},
for any $i$ and $j \ge J(\varepsilon)$, we see that $a_{ij} \le (1+2\varepsilon)b_{\varepsilon,i}$
and so $E^{n(j)}_{\{a_{ij}\}_i} \prec E^{n(j)}_{\{(1+2\varepsilon)b_{\varepsilon,i}\}}
= (1+2\varepsilon)E^{n(j)}_{\{b_{\varepsilon,i}\}}$.
Lemma \ref{lem:2} implies that $\cP {E^{n(j)}_{\{\sqrt{n(j) - 1} \, b_{\varepsilon,i}\}}}$ converges weakly to
$\cP {\Gamma^\infty_{\{b_{\varepsilon,i}^2\}}}$ as $j\to\infty$.
Therefore, $\cP_\infty$ is contained in $(1+2\varepsilon)\cP {\Gamma^\infty_{\{b_{\varepsilon,i}^2\}}}$ for any $\varepsilon > 0$.
Since $b_{\varepsilon,i} \le a_i$, we see that
$\cP_\infty$ is contained in $(1+2\varepsilon)\cP {\Gamma^\infty_{\{a_i^2\}}}$
for any $\varepsilon > 0$.
This proves the lemma in this case.

We next assume $a_\infty = 0$.
For any $\varepsilon > 0$ there is a number $I(\varepsilon)$
such that
\begin{equation} \label{eq:0}
a_i < \varepsilon \qquad\text{for any $i \ge I(\varepsilon)$.}
\end{equation}
We may assume that $I(\varepsilon) = i_0 + 1$ if $i_0 < \infty$,
where $i_0 := \sup\{\;i \mid a_i > 0\;\}$.
Also, there is a number $J(\varepsilon)$ such that
\begin{align}
\label{eq:i}
a_{I(\varepsilon),j} &< a_{I(\varepsilon)} + \varepsilon \qquad\text{for any $j \ge J(\varepsilon)$;}\\
\label{eq:ii}
a_{ij} &< (1+\varepsilon)a_i \qquad\text{for any $i < I(\varepsilon)$ and $j \ge J(\varepsilon)$.}
\end{align}
It follows from \eqref{eq:0} and \eqref{eq:i} that 
\begin{equation} \label{eq:b2}
a_{ij} \le a_{I(\varepsilon),j} < a_{I(\varepsilon)} + \varepsilon < 2\varepsilon
\quad\text{for any $i \ge I(\varepsilon)$ and $j \ge J(\varepsilon)$.}
\end{equation}
Let
\[
b_{\varepsilon,i} :=
\begin{cases}
  (1+\varepsilon)a_i &\text{if $i < I(\varepsilon)$,}\\
  2\varepsilon &\text{if $i \ge I(\varepsilon)$.}
\end{cases}
\]
From \eqref{eq:ii} and \eqref{eq:b2}, we have $a_{ij} < b_{\varepsilon,i}$ for any $i$ and $j \ge J(\varepsilon)$,
and so $E^{n(j)}_{\{a_{ij}\}_i} \prec E^{n(j)}_{\{b_{\varepsilon,i}\}}$ for $j \ge J(\varepsilon)$.
Lemma \ref{lem:2} implies that $\cP {E^{n(j)}_{\{\sqrt{n(j) - 1} \, b_{\varepsilon,i}\}}}$
converges weakly to $\cP {\Gamma^\infty_{\{b_{\varepsilon,i}^2\}}}$ as $j\to\infty$.
Therefore, $\cP_\infty$ is contained in $\cP {\Gamma^\infty_{\{b_{\varepsilon,i}^2\}}}$
for any $\varepsilon > 0$.
Let $k$ be any number with $k \ge I(\varepsilon)$.
The Gaussian space $\Gamma^k_{\{b_{\varepsilon,i}^2\}}$ is mm-isomorphic to
the $l_2$-product of
$\Gamma^{I(\varepsilon)-1}_{\{(1+\varepsilon)^2a_i^2\}}$
and $\Gamma^{k-I(\varepsilon)+1}_{\{(2\varepsilon)^2\}}$.
It follows from the Gaussian isoperimetry that
\[
\ObsDiam(\Gamma^{k-I(\varepsilon)+1}_{\{(2\varepsilon)^2\}})
= \inf_{\kappa > 0} \max\{2\varepsilon\diam(\gamma^1_{1^2};1-\kappa),\kappa\} =: \tau(\varepsilon),
\]
which tends to zero as $\varepsilon \to 0+$.
If $\tau(\varepsilon) < 1/2$, then, by \cite{Sy:mmg}*{Proposition 7.32},
\[
\rho( \cP {\Gamma^k_{\{b_{\varepsilon,i}^2\}}},
\cP {\Gamma^{I(\varepsilon)-1}_{\{(1+\varepsilon)^2a_i^2\}}} )
\le \dconc(\Gamma^k_{\{b_{\varepsilon,i}^2\}},\Gamma^{I(\varepsilon)-1}_{\{(1+\varepsilon)^2a_i^2\}})
\le \tau(\varepsilon).
\]
Taking the limit as $k\to\infty$ yields
\[
\rho( \cP {\Gamma^\infty_{\{b_{\varepsilon,i}^2\}}},
\cP {\Gamma^{I(\varepsilon)-1}_{\{(1+\varepsilon)^2a_i^2\}}} )
\le \tau(\varepsilon).
\]
There is a sequence $\{\varepsilon(l)\}$, $l=1,2,\dots$,
of positive real numbers tending to zero
such that $\cP {\Gamma^\infty_{\{b_{\varepsilon(l),i}^2\}}}$
converges weakly to
a pyramid $\cP_\infty'$ as $l\to\infty$.
$\cP_\infty'$ contains $\cP_\infty$ and $\cP {\Gamma^{I(\varepsilon(l))-1}_{\{(1+\varepsilon(l))^2a_i^2\}}}$ converges weakly to $\cP_\infty'$ as $l\to\infty$.
Since $\cP {\Gamma^{I(\varepsilon(l))-1}_{\{(1+\varepsilon(l))^2a_i^2\}}}$
is contained in $\cP {\Gamma^\infty_{\{(1+\varepsilon(l))^2a_i^2\}}}$
and since $\cP {\Gamma^\infty_{\{(1+\varepsilon(l))^2a_i^2\}}} = (1+\varepsilon(l))\cP {\Gamma^\infty_{\{a_i^2\}}}$ converges weakly to $\cP {\Gamma^\infty_{\{a_i^2\}}}$ as $l\to\infty$,
the pyramid $\cP_\infty'$ is contained in $\cP {\Gamma^\infty_{\{a_i^2\}}}$,
so that $\cP_\infty$ is contained in $\cP {\Gamma^\infty_{\{a_i^2\}}}$.
This completes the proof.
\end{proof}

\begin{proof}[Proof of Theorem {\rm\ref{thm:ellipsoid}}]
Suppose that $\cP {E^{n(j)}_{\{\sqrt{n(j) - 1} \, a_{ij}\}_i}}$ does not converge weakly to
$\cP {\Gamma^\infty_{\{a_i^2\}}}$ as $j \to \infty$.
Then, taking a subsequence of $\{j\}$ we may assume that
$\cP {E^{n(j)}_{\{\sqrt{n(j) - 1} \, a_{ij}\}_i}}$ converges weakly to a pyramid $\cP_\infty$
different from $\cP {\Gamma^\infty_{\{a_i^2\}}}$, which contradicts
Lemmas \ref{lem:supset} and \ref{lem:subset}.
Thus, $\cP {E^{n(j)}_{\{\sqrt{n(j) - 1} \, a_{ij}\}_i}}$ converges weakly to
$\cP {\Gamma^\infty_{\{a_i^2\}}}$ as $j \to \infty$.

As is mentioned in Subsection \ref{ssec:Gaussian},
the infinite-dimensional Gaussian space $\Gamma^\infty_{\{a_i^2\}}$
is well-defined as an mm-space if and only if
$\{a_i\}$ is an $l^2$-sequence, only in which case
the above sequence of (solid) ellipsoids becomes a convergent
sequence in the concentration topology.

Assume that $a_i$ converges to zero as $i\to\infty$.
It is well-known that the Ornstein-Uhlenbeck operator
(or the drifted Laplacian) on $\Gamma^1_{a^2}$ has compact resolvent
and spectrum $\{ka^{-2} \mid k = 0,1,2\dots\}$.
Thus, the same proof as in \cite{Sy:mmg}*{Corollary 7.35}
yields that $\Gamma^n_{\{a_i^2\}}$ asymptotically (spectrally) concentrates
to $\Gamma^\infty_{\{a_i\}}$.

Conversely, we assume that $a_i$ is bounded away from zero
and set $\underline{a} := \inf_i a_i$.
Applying \cite{Sy:mmg}*{Proposition 7.37} yields
that $\cP\Gamma^\infty_{\{\underline{a}^2\}}$ is not concentrated.
Since $\cP\Gamma^\infty_{\{a_i^2\}}$ contains
$\cP\Gamma^\infty_{\{\underline{a}^2\}}$,
the pyramid $\cP\Gamma^\infty_{\{a_i^2\}}$ is not concentrated,
which implies that
$E^{n(j)}_{\{\sqrt{n(j) - 1} \, a_{ij}\}_i}$ does not asymptotically concentrate
(see Theorem \ref{thm:asymp-conc}).

This completes the proof of the theorem.
\end{proof}

Let us next consider the convergence of the Gaussian spaces.

\begin{prop}
Let $\{a_{ij}\}$, $i=1,2,\dots,n(j)$, $j=1,2,\dots$,
be a sequence of nonnegative real numbers.
If $\sup_i a_{ij}$ diverges to infinity as $j\to\infty$,
then $\Gamma^{n(j)}_{\{a_{ij}^2\}}$ infinitely dissipates.
\end{prop}

\begin{proof}
Exchanging the coordinates, we assume that
$a_{1j}$ diverges to infinity as $j\to\infty$.
Since $\Gamma^1_{a_{1j}^2}$ is dominated by $\Gamma^{n(j)}_{\{a_{ij}^2\}}$,
we have
\[
\ObsDiam(\Gamma^{n(j)}_{\{a_{ij}^2\}};-\kappa)
\ge \diam(\Gamma^1_{a_{1j}^2};1-\kappa) \to \infty
\quad\text{as $j\to\infty$}. 
\]
This together with Proposition \ref{prop:dissipate} completes the proof.
\end{proof}

In a similar way as in the proof of Theorem \ref{thm:ellipsoid},
we obtain the following.

\begin{cor}
Let $\{a_{ij}\}$ satisfy {\rm(A0)--(A3)}.
Then, $\Gamma^{n(j)}_{\{a_{ij}^2\}}$ converges weakly to $\cP {\Gamma^\infty_{\{a_i^2\}}}$ as $j\to\infty$.
This convergence becomes a convergence in the concentration topology if and only if $\{a_i\}$ is an $l^2$-sequence.
Moreover, this convergence becomes an asymptotic concentration
if and only if $\{a_i\}$ converges to zero.

\end{cor}

\begin{proof}
Suppose that $\Gamma^{n(j)}_{\{a_{ij}^2\}}$ does not converge weakly to
 $\cP {\Gamma^\infty_{\{a_i^2\}}}$ as $j\to\infty$.
Then there is a subsequence of $\{\cP {\Gamma^{n(j)}_{\{a_{ij}^2\}}}\}_j$
that converges weakly to a pyramid $\cP_\infty$ different from $\cP {\Gamma^\infty_{\{a_i^2\}}}$.
We write such a subsequence by the same notation $\{\cP {\Gamma^{n(j)}_{\{a_{ij}^2\}}\}_j}$.

Since $\Gamma^k_{\{a_{ij}^2\}}$ is dominated by $\Gamma^{n(j)}_{\{a_{ij}^2\}}$ for $k \le n(j)$
and $\Gamma^k_{\{a_{ij}^2\}}$ converges weakly to $\Gamma^k_{\{a_i^2\}}$ as $j\to\infty$,
we see that $\Gamma^k_{\{a_i^2\}}$ belongs to $\cP_\infty$ for any $k$, so that
\[
\cP_\infty \supset \cP {\Gamma^\infty_{\{a_i^2\}}}.
\]

We prove $\cP_\infty \subset \cP {\Gamma^\infty_{\{a_i^2\}}}$ in the case of $a_\infty > 0$,
where $a_\infty := \lim_{i\to\infty} a_i$.
Under $a_\infty > 0$, the same discussion as in the proof of Lemma \ref{lem:supset}
proves that there are two numbers $I(\varepsilon)$ and $J(\varepsilon)$
for any $\varepsilon > 0$ such that \eqref{eq:aijb1} and \eqref{eq:aijb2} both hold.
We therefore see that, for any $k$ and $j \ge J(\varepsilon)$,
$\Gamma^k_{\{a_{ij}^2\}}$ is dominated by $(1+\varepsilon)\Gamma^k_{\{a_i^2\}}$,
and so $\cP {\Gamma^{n(j)}_{\{a_{ij}^2\}_i}} \subset (1+\varepsilon)\cP {\Gamma^\infty_{\{a_i^2\}}}$.
This proves $\cP_\infty \subset \cP {\Gamma^\infty_{\{a_i^2\}}}$.

We next prove $\cP_\infty \subset \cP {\Gamma^\infty_{\{a_i^2\}}}$ in the case of $a_\infty = 0$.
Let $b_{\varepsilon,i}$ be as in Lemma \ref{lem:subset}.
The discussion in the proof of Lemma \ref{lem:subset}
yields that $a_{ij} < b_{\varepsilon,i}$ for any $i$ and for every sufficiently large $j$,
which implies $\cP_\infty \subset \cP {\Gamma^\infty_{\{b_{\varepsilon,i}^2\}}}$.
We obtain $\cP_\infty \subset \cP {\Gamma^\infty_{\{a_i^2\}}}$
in the same way as in the proof of Lemma \ref{lem:subset}.
The weak convergence of $\Gamma^{n(j)}_{\{a_{ij}^2\}}$
to $\cP {\Gamma^\infty_{\{b_{\varepsilon,i}^2\}}}$ has been proved.

The rest is identical to the proof of Theorem \ref{thm:ellipsoid}.
This completes the proof.
\end{proof}

\section{Box convergence of ellipsoids}

The main purpose of this section is to prove Theorem \ref{thm:box-conv}.

Let us first prove the weak convergence of $e^n_{\{\sqrt{n-1}\,a_{ij}\}}$
if $\{a_{ij}\}$ $l^2$-converges.

\begin{lem} \label{lem:eg}
Let $\mathcal{A}$ be a family of sequences of positive real numbers
such that
\[
\sup_{\{a_i\} \in \mathcal{A}} \sum_{i=1}^\infty a_i^2 < +\infty.
\]
Then we have
\begin{align}
&\lim_{n\to\infty} \sup_{\{a_i\} \in \mathcal{A}}
\dP(\epsilon^n_{\{\sqrt{n-1}\,a_i\}},\gamma^n_{\{a_i^2\}}) = 0,\tag{1}\\
&\limsup_{n\to\infty} \sup_{\{a_i\} \in \mathcal{A}}
W_2(\sigma_{\{\sqrt{n - 1} \, a_i\}}^{n - 1},\gamma_{\{a_i^2\}}^n)^2
\le \sqrt{2} e^{-1} \sup_{\{a_i\} \in \mathcal{A}}
\sum_{i=k+1}^\infty a_i^2 \tag{2}
\end{align}
for any positive integer $k$.
\end{lem}

\begin{proof}
We prove (1).
Let $r(x) := \|L^{-1}(x)\|$ as in Section \ref{sec:weak-conv}.
Take any real number $\theta$ with $0 < \theta < 1$
and fix it.
Let us consider the normalization of the measures
$\epsilon^n_{\{\sqrt{n-1}\,a_i\}}|_{r^{-1}([\,\theta\sqrt{n-1},\sqrt{n-1}\,])}$
and
$\gamma^n_{\{a_i^2\}}|_{r^{-1}([\,\theta\sqrt{n-1},\theta^{-1}\sqrt{n-1}\,])}$,
which we denote by $\epsilon^n_\theta$ and $\gamma^n_\theta$,
respectively.  Set
\begin{align*}
v_{\theta,n} &:= \epsilon_{\sqrt{n-1}}^n(\{\,x \in \R^n \mid
\theta\sqrt{n-1} \le \|x\| \le \sqrt{n-1}\,\}),\\
w_{\theta,n} &:= \gamma_{\{1^2\}}^n(\{\,x \in \R^n \mid
\theta\sqrt{n-1} \le \|x\| \le \theta^{-1}\sqrt{n-1}\,\}).
\end{align*}
We remark that
\begin{align*}
v_{\theta,n} &= \epsilon_{\{\sqrt{n-1}\,a_i\}}^n(r^{-1}([\,\theta\sqrt{n-1},\sqrt{n-1}\,])),\\
w_{\theta,n} &= \gamma_{\{a_i^2\}}^n(r^{-1}([\,\theta\sqrt{n-1},\theta^{-1}\sqrt{n-1}\,])),\\
&\lim_{n\to\infty} v_{\theta,n} = \lim_{n\to\infty} w_{\theta,n} = 1.
\end{align*}
It then holds that
\begin{align}
\label{eq:eg1}
\dP(\epsilon^n_{\{\sqrt{n-1}\,a_i\}},\epsilon^n_\theta)
&\le \dTV(\epsilon^n_{\{\sqrt{n-1}\,a_i\}},\epsilon^n_\theta)
= 1-v_{\theta,n},\\
\label{eq:eg2}
\dP(\gamma^n_{\{a_i^2\}},\gamma^n_\theta)
&\le \dTV(\gamma^n_{\{a_i^2\}},\gamma^n_\theta)
= 1-w_{\theta,n}.
\end{align}
To estimate $\dP(\epsilon^n_\theta,\gamma^n_\theta)$,
we define a transport map, say $\psi$, from $\gamma^n_\theta$ to $\epsilon^n_\theta$ in the same manner as for $\varphi^{\mathcal{E}}$
in Section \ref{sec:weak-conv}, which is expressed as
\[
\psi(x) = \frac{\tilde{R}}{r} x,
\quad x \in r^{-1}([\,\theta\sqrt{n-1},\theta^{-1}\sqrt{n-1}\,]),
\]
where $\tilde{R}$ is the function of variable $r \in [\,\theta\sqrt{n-1},\theta^{-1}\sqrt{n-1}\,]$ defined by
\[
\theta\sqrt{n-1} \le \tilde{R} \le \sqrt{n-1}
\quad\text{and}\quad
\gamma_\theta^n(B_r(o)) = \epsilon_\theta^n(B_R(o)).
\]
Since $\theta^2 \le \tilde{R}/r \le \theta^{-1}$, we have
\begin{align*}
W_2(\epsilon^n_\theta,\gamma^n_\theta)^2
&\le \int_{\R^n} \|\,\psi(x)-x\,\|^2 \,d\gamma^n_\theta(x)
= \int_{\R^n} \left( \frac{\tilde{R}}{r} - 1\right)^2 \|x\|^2 \,d\gamma^n_\theta(x)\\
&\le \max\{(1-\theta)^2,(\theta^2-1)^2\}
\int_{\R^n} \|x\|^2 \, d\gamma^n_\theta(x)\\
&\le \frac{\max\{(1-\theta)^2,(\theta^2-1)^2\}} 
{\gamma^n_{\{a_i^2\}}(r^{-1}([\,\theta\sqrt{n-1},\theta^{-1}\sqrt{n-1}\,]))}
 \int_{\R^n} \|x\|^2 \, d\gamma^n_{\{a_i^2\}}(x)\\
&= \frac{\max\{(1-\theta)^2,(\theta^2-1)^2\}}{w_{\theta,n}} \sum_{i=1}^n a_i^2,
\end{align*}
which together with \eqref{eq:eg1} and \eqref{eq:eg2} implies
\begin{align*}
&\dP(\epsilon^n_{\{\sqrt{n-1}\,a_i\}},\gamma^n_{\{a_i^2\}})\\
&\le 2 - v_{\theta,n} - w_{\theta,n}
+ \left( \frac{\max\{(1-\theta)^2,(\theta^2-1)^2\}}{w_{\theta,n}} \sum_{i=1}^n a_i^2 \right)^{\frac{1}{4}}
\end{align*}
and hence
\begin{align*}
&\limsup_{n\to\infty} \sup_{\{a_i\} \in \mathcal{A}} \dP(\epsilon^n_{\{\sqrt{n-1}\,a_i\}},\gamma^n_{\{a_i^2\}})\\
&\le \left( \max\{(1-\theta)^2,(\theta^2-1)^2\} \sup_{\{a_i\} \in \mathcal{A}}\sum_{i=1}^\infty a_i^2 \right)^{\frac{1}{4}}
\to 0 \quad\text{as $\theta \to 1+$}.
\end{align*}
This proves (1).

We prove (2).
Using the transport map $\varphi^{\mathcal{S}}$
from $\gamma_{\{a_i^2\}}^n$ to $\sigma_{\{\sqrt{n-1}\,a_i\}}^{n - 1}$
in Section \ref{sec:weak-conv}, we have
\begin{align*}
& W_2(\sigma_{\{\sqrt{n - 1} \, a_i\}}^{n - 1},\gamma_{\{a_i^2\}}^n)^2
\le \int_{\R^n} \left\|\, z - \varphi^{\mathcal{S}}(z) \,\right\|^2 \, d\gamma^n_{\{a_i^2\}}(z)\\
&= \int_{S^{n - 1}(1)} \sum_{i=1}^n a_i^2 x_i^2 \, d\sigma^{n - 1}(x)
\cdot \frac{1}{I_{n - 1}} \int_0^\infty (r-\sqrt{n - 1})^2 r^{n - 1} e^{-r^2/2} \, dr,
\end{align*}
where $I_m := \int_0^\infty t^m e^{-t^2/2}\,dt$.
We see in the proof of \cite{Sy:mmg}*{Lemma 7.41}
that $r^m e^{-r^2/2} \le m^{m/2} e^{-m/2} e^{-(r-\sqrt{m})^2/2}$
and also that $I_m \sim \sqrt{\pi} (m-1)^{m/2} e^{-(m-1)/2}$.
Therefore,
\begin{align*}
&\frac{1}{I_{n - 1}} \int_0^\infty (r-\sqrt{n - 1})^2 r^{n - 1} e^{-r^2/2} \, dr
\le \frac{1}{I_{n - 1}} \sqrt{2\pi} (n - 1)^{(n - 1)/2} e^{-(n - 1)/2}\\
&\sim \frac{\sqrt{2} e^{-1/2}}{(1-1/(n - 1))^{(n - 1)/2}} \longrightarrow \sqrt{2}e^{-1}
\ \text{as}\ n\to\infty.
\end{align*}
For any $\varepsilon > 0$ and $k$ with $1 \le k \le n - 1$, let
$S_{k,\varepsilon}^{n - 1} := \{\,x \in S^{n - 1}(1) \mid
|x_i| < \varepsilon\ \text{for}\ i=1,2,\dots,k\,\}$.
Then,
\begin{align*}
\int_{S^{n - 1}(1) \setminus S_{k,\varepsilon}^{n - 1}} \sum_{i=1}^n a_i^2 x_i^2 \, d\sigma^{n - 1}(x)
&\le \sigma^{n - 1}(S^{n - 1}(1) \setminus S_{k,\varepsilon}^{n - 1}) \sum_{i=1}^n a_i^2,\\
\int_{S_{k,\varepsilon}^{n - 1}} \sum_{i=1}^n a_i^2 x_i^2 \, d\sigma^{n - 1}(x)
&\le \varepsilon^2 \sum_{i=1}^k a_i^2 + \sum_{i=k+1}^n a_i^2,
\end{align*}
which imply
\begin{align*}
&\sup_{\{a_i\} \in \mathcal{A}} 
\int_{S^{n - 1}(1)} \sum_{i=1}^n a_i^2 x_i^2 \, d\sigma^{n - 1}(x)\\
&\le (\sigma^{n - 1}(S^{n - 1}(1) \setminus S_{k,\varepsilon}^{n - 1}) + \varepsilon^2)
\sup_{\{a_i\} \in \mathcal{A}} \sum_{i=1}^\infty a_i^2
+ \sup_{\{a_i\} \in \mathcal{A}} \sum_{i=k+1}^\infty a_i^2.
\end{align*}
Since $\sigma^{n - 1}(S^{n - 1}(1) \setminus S_{k,\varepsilon}^{n - 1})$ tends to zero
as $n \to\infty$, we obtain (2).
This completes the proof.
\end{proof}

\begin{prop} \label{prop:box-conv}
Let $\{a_{ij}\}$, $i=1,2,\dots,n(j)$, $j=1,2,\dots$, be a sequence of 
positive real numbers,
where $\{n(j)\}$, $j=1,2,\dots$, is a sequence of positive integers
divergent to infinity.
Let $\{a_i\}$, $i=1,2,\dots$, be an $l_2$-sequence of nonnegative real numbers.
We assume
\[
\lim_{j\to\infty} \sum_{i=1}^{n(j)} (a_{ij} - a_i)^2 = 0.
\]
Then we have
\begin{enumerate}
\item $\epsilon_{\{\sqrt{n(j) - 1}\,a_{ij}\}_i}^{n(j)}$ converges weakly to
$\gamma_{\{a_i^2\}_i}^\infty$ as $j\to\infty${\rm;}
\item $\sigma_{\{\sqrt{n(j) - 1}\,a_{ij}\}_i}^{n(j) - 1}$ converges to
$\gamma_{\{a_i^2\}_i}^\infty$ in the $2$-Wasserstein metric
as $j\to\infty$.
\end{enumerate}
In particular, $E^{n(j)}_{\{\sqrt{n(j) - 1}\,a_{ij}\}_i}$ box converges to
$\Gamma^\infty_{\{a_i^2\}_i}$ as $j\to\infty$.
\end{prop}

\begin{proof}
We first prove (2).
We set $a_{ij} := 0$ for $i \ge n(j)+1$.
Note that the assumption implies the $l_2$-convergence
of $\{a_{ij}\}$ to $\{a_i\}$ as $j\to\infty$.
Lemma \ref{lem:eg}(2) implies
\begin{align*}
&\limsup_{j\to\infty}
W_2(\sigma_{\{\sqrt{n(j) - 1}\,a_{ij}\}_i}^{n(j) - 1},\gamma_{\{a_{ij}^2\}_i}^{n(j)})^2\\
&\le \sqrt{2}e^{-1} \limsup_{j\to\infty} \sum_{i=k+1}^\infty a_{ij}^2
= \sqrt{2}e^{-1} \sum_{i=k+1}^\infty a_i^2
\longrightarrow 0 \ \text{as}\ k\to\infty.
\end{align*}
Gelbrich's formula \cite{Gel:formulaW2} tells us that
\[
W_2(\gamma_{\{a_{ij}^2\}_i}^{n(j)},\gamma_{\{a_i^2\}}^\infty)^2
= \sum_{i=1}^\infty (a_{ij} - a_i)^2 \longrightarrow 0\ \text{as}\ j\to\infty.
\]
By a triangle inequality, we obtain (2).

(1) is proved in the same way by using Lemma \ref{lem:eg}(1) and
by remarking $\dP^2 \le W_2$.
This completes the proof.
\end{proof}

\begin{lem} \label{lem:b}
Let $\{b_{ij}\}$, $i=1,2,\dots,n(j)$, $j=1,2,\dots$, be a sequence of 
positive real numbers,
where $\{n(j)\}$, $j=1,2,\dots$, is a sequence of positive integers divergent to infinity.
If $\sum_{i=1}^{n(j)} b_{ij}^2$ converges to
a positive real number as $j\to\infty$,
then $\{e^{n(j)}_{\{\sqrt{n(j)-1}\,b_{ij}\}}\}$ has no subsequence
converging weakly to the Dirac measure $\delta_o$ at the origin $o$ in $H$,
where we embed the {\rm(}solid{\rm)} ellipsoids
$E^{n(j)}_{\{\sqrt{n(j)-1}\,b_{ij}\}} \subset \R^{n(j)}$
into the Hilbert space $H$ naturally and consider
$e_{\{\sqrt{n(j) - 1}\,b_{ij}\}_i}^{n(j)}$ as Borel
probability measures on $H$.
\end{lem}

\begin{proof}
We first prove the lemma for $\sigma^{n(j)-1}_{\{\sqrt{n(j)-1}\,b_{ij}\}}$.
It holds that
\begin{align*}
W_2(\sigma_{\{\sqrt{n(j) - 1}\,b_{ij}\}_i}^{n(j) - 1},\delta_o)^2
&= \int_{\mathcal{S}_{\{\sqrt{n(j) - 1}\,b_{ij}\}_i}^{n(j) - 1}} \|y\|^2
\, d\sigma_{\{\sqrt{n(j) - 1}\,b_{ij}\}_i}^{n(j) - 1}(y)\\
&= \int_{S^{n - 1}(1)} \sum_{i=1}^{n(j)} (n(j) - 1) b_{ij}^2 x_i^2
\, d\sigma^{n(j) - 1}(x)\\
&= \left(\sum_{i=1}^{n(j)} b_{ij}^2\right)
(n(j) - 1) \int_{S^{n - 1}(1)} x_1^2 \, d\sigma^{n(j) - 1}(x).
\end{align*}
It follows from the Maxwell-Boltzmann distribution law that
\[
\lim_{j\to\infty} (n(j) - 1)\int_{S^{n - 1}(1)} x_1^2 \, d\sigma^{n(j) - 1}(x) = 1.
\]
We therefore have
\[
\lim_{j\to\infty} W_2(\sigma_{\{\sqrt{n(j) - 1}\,b_{ij}\}_i}^{n(j) - 1},\delta_o)^2
= \lim_{j\to\infty} \sum_{i=1}^{n(j)} b_{ij}^2 > 0,
\]
so that $\{\sigma_{\{\sqrt{n(j) - 1}\,b_{ij}\}_i}^{n(j) - 1}\}$ has no
subsequence converging weakly to $\delta_o$.

We prove the lemma for $\epsilon^{n(j)}_{\{\sqrt{n(j)-1}\,b_{ij}\}}$.
Applying Lemma \ref{lem:eg}(1) yields that
\[
\lim_{j\to\infty} \dP(\epsilon^{n(j)}_{\{\sqrt{n(j)-1}\,b_{ij}\}},\gamma^{n(j)}_{\{b_{ij}^2\}}) = 0.
\]
We also have
\[
\lim_{j\to\infty} W_2(\gamma^{n(j)}_{\{b_{ij}^2\}},\delta_o)^2
= \lim_{j\to\infty} \int_{\R^n} \|x\|^2 \, d\gamma^{n(j)}_{\{b_{ij}^2\}}(x)
= \lim_{j\to\infty} \sum_{i=1}^{n(j)} b_{ij}^2 > 0
\]
and hence $\{\gamma^{n(j)}_{\{b_{ij}^2\}}\}$ does not have
a subsequence converging weakly to $\delta_o$
and so does $\{\epsilon^{n(j)}_{\{\sqrt{n(j)-1}\,b_{ij}\}}\}$.
This completes the proof of the lemma.
\end{proof}

The following lemma is the special case of Theorem \ref{thm:box-conv}
where the limit is a one-point mm-space.

\begin{lem} \label{lem:box-conv}
Let $\{a_{ij}\}$, $i=1,2,\dots,n(j)$, $j=1,2,\dots$, be a sequence of 
positive real numbers,
where $\{n(j)\}$, $j=1,2,\dots$, is a sequence of positive integers divergent to infinity.
We assume that
\begin{align}
\tag{i} &\lim_{j\to\infty} a_{ij} = 0 \quad\text{for any $i$},\\
\tag{ii} &\liminf_{j\to\infty} \sum_{i=1}^{n(j)} a_{ij}^2 > 0.
\end{align}
Then, there exists no box convergent subsequence of
$\{E_{\{\sqrt{n(j) - 1}\,a_{ij}\}_i}^{n(j)}\}$.
\end{lem}

\begin{proof}
Let $\{a_{ij}\}$ be a sequence as in the assumption of the theorem.
Sorting $\{a_{ij}\}$ in ascending order in $i$,
we may assume that $a_{ij}$ is monotone nonincreasing in $i$
for each $j$.
We suppose that $\{E_{\{\sqrt{n(j) - 1}\,a_{ij}\}_i}^{n(j)}\}$
has a box convergent subsequence, for which we use
the same notation.
Then, by (i) and Theorem \ref{thm:ellipsoid},
the box limit of $\{E_{\{\sqrt{n(j) - 1}\,a_{ij}\}_i}^{n(j)}\}$
is mm-isomorphic to a one-point mm-space.
We set
\[
A_j := \left( \sum_{i=1}^{n(j)} a_{ij}^2 \right)^{\frac12}
\quad\text{and}\quad
b_{ij} := \frac{a_{ij}}{\max\{A_j,1\}}.
\]
Since $b_{ij} \le a_{ij}$,
we see that $E_{\{\sqrt{n(j) - 1}\,b_{ij}\}_i}^{n(j)}$
is dominated by $E_{\{\sqrt{n(j) - 1}\,a_{ij}\}_i}^{n(j)}$,
so that $E_{\{\sqrt{n(j) - 1}\,b_{ij}\}_i}^{n(j)}$ box converges
to a one-point mm-space as $j\to\infty$.
We remark that
\[
\liminf_{j\to\infty} \sum_{i=1}^{n(j)} b_{ij}^2 > 0
\quad\text{and}\quad
\sum_{i=1}^{n(j)} b_{ij}^2 \le 1.
\]
Taking a subsequence again, we assume that
$\sum_{i=1}^{n(j)} b_{ij}^2$ converges to a positive real number
as $j\to\infty$.
Applying Lemma \ref{lem:b} yields that
$\{e^{n(j)}_{\{\sqrt{n(j)-1}\,b_{ij}\}}\}$ has no subsequence
converging weakly to $\delta_o$ in $H$.
Since $E_{\{\sqrt{n(j) - 1}\,b_{ij}\}_i}^{n(j)}$
box converges to a one-point mm-space, say $*$, as $j\to\infty$,
Lemma \ref{lem:box-mmiso} implies that
there is a sequence of $\varepsilon_j$-mm-isomorphisms
$f_j : E_{\{\sqrt{n(j) - 1}\,b_{ij}\}_i}^{n(j)} \to *$ with $\varepsilon_j \to 0+$
as $j\to\infty$.  A nonexceptional domain of $f_j$ has
$e^{n(j)}_{\{\sqrt{n(j) - 1}\,b_{ij}\}}$-measure at least $1-\varepsilon_j$
and diameter at most $\varepsilon_j$.
There is a closed metric ball $B_j \subset H$
of radius $\varepsilon_j$
that contains the nonexceptional domain of $f_j$.
Note that $e_{\{\sqrt{n(j) - 1}\,b_{ij}\}_i}^{n(j)}(B_j) \ge 1-\varepsilon_j
\to 1$ as $j\to\infty$.
If $B_j$ were to contain the origin $o$ of $H$ for infinitely many $j$,
then a subsequence of $\{e_{\{\sqrt{n(j) - 1}\,b_{ij}\}_i}^{n(j)}\}$
would converge weakly to $\delta_o$, which is a contradiction.
Thus, all but finitely many $B_j$ do not contain the origin of $H$,
and $B_j$ do not intersect $-B_j$ for any such $B_j$.
Since $e_{\{\sqrt{n(j) - 1}\,b_{ij}\}_i}^{n(j)}$ is centrally symmetric with respect
to the origin, we see that $e_{\{\sqrt{n(j) - 1}\,a_{ij}\}_i}^{n(j)}(-B_j)
= e_{\{\sqrt{n(j) - 1}\,a_{ij}\}_i}^{n(j)}(B_j) \ge 1-\varepsilon_j$,
which is a contradiction if $j$ is large enough.
This completes the proof.
\end{proof}

\begin{lem} \label{lem:box-dom}
Let $\{X_n\}$, $n=1,2,\dots$, be a box convergent sequence of mm-spaces
and $\{Y_n\}$, $n=1,2,\dots$, a sequence of mm-spaces
with $Y_n \prec X_n$.
Then, $\{Y_n\}$ has a box convergent subsequence.
\end{lem}

\begin{proof}
The lemma follows from \cite{Sy:mmg}*{Lemma 4.28}.
\end{proof}

\begin{proof}[Proof of Theorem {\rm\ref{thm:box-conv}}]
We assume (A0)--(A3).

The `if' part follows from Proposition \ref{prop:box-conv}.

We prove the `only if' part.
Suppose that $\{E^{n(j)}_{\{\sqrt{n(j) - 1}\,a_{ij}\}_i}\}$
is box convergent
and that $\{a_{ij}\}_i$ does not $l^2$-converge
to $\{a_i\}$ as $j\to\infty$.
We first prove that $\{a_i\}$ is an $l^2$-sequence.
This is because, if not, then, by Theorem \ref{thm:ellipsoid},
the weak limit of $\{E^{n(j)}_{\{\sqrt{n(j) - 1}\,a_{ij}\}_i}\}$ is not an mm-space, which is a contradiction to the box convergence.
Replacing $\{a_{ij}\}_i$ with a subsequence with respect to the index $j$,
we assume that $\lim_{j\to\infty} \sum_{i=1}^{n(j)} a_{ij}^2$ exists
in $[\,0,+\infty\,]$.
We prove
\begin{equation} \label{eq:sum}
\lim_{j\to\infty} \sum_{i=1}^{n(j)} a_{ij}^2 > \sum_{i=1}^\infty a_i^2.
\end{equation}
In fact, if the left-hand side of \eqref{eq:sum} is infinity,
then this is clear.
If not, the Banach-Alaoglu theorem tells us the existence of
an $l^2$-weakly convergent subsequence of $\{a_{ij}\}$.
Since $\{a_{ij}\}_i$ does not converge to $\{a_i\}$ $l^2$-strongly as $j\to\infty$,
we obtain \eqref{eq:sum}.

Take a real number $\varepsilon_0$ in such a way that
\[
0 < \varepsilon_0 < \lim_{j\to\infty} \sum_{i=1}^{n(j)} a_{ij}^2 - \sum_{i=1}^\infty a_i^2.
\]
Setting
\[
a_{ijk} :=
\begin{cases}
  a_{kj} &\text{if $i \le k$},\\
  a_{ij} &\text{if $i \ge k+1$},
\end{cases}
\]
we have
\begin{align*}
\lim_{j\to\infty} \sum_{i=1}^{n(j)} a_{ijk}^2
&= \lim_{j\to\infty} \left( \sum_{i=1}^k a_{ijk}^2 + \sum_{i=k+1}^{n(j)} a_{ijk}^2 \right)
= k a_k^2 + \lim_{j\to\infty} \sum_{i=k+1}^{n(j)} a_{ij}^2\\
&= k a_k^2 + \lim_{j\to\infty} \sum_{i=1}^{n(j)} a_{ij}^2 - \sum_{i=1}^k a_i^2
> \varepsilon_0.
\end{align*}
Thus, for any positive integer $k$ there is $j(k)$
such that
\[
\sum_{i=1}^{n(j(k))} a_{ij(k)k}^2 > \varepsilon_0
\quad\text{and}\quad
|\, a_{kj(k)} - a_k\,| < \frac{1}{k}.
\]
Letting $b_{ik} := a_{ij(k)k}$, we observe the following.
\begin{itemize}
\item $b_{ik} \le a_{ij(k)}$ for any $i$ and $k$.
\item $b_{ik}$ is monotone nonincreasing in $i$ for each $k$.
\item $b_{1k} = a_{1j(k)k} = a_{kj(k)} < a_k + 1/k \to 0$ as $k\to\infty$.
\item $\sum_{i=1}^{n(j(k))} b_{ik}^2 > \varepsilon_0 > 0$ for any $k$.
\end{itemize}
Consider $E_k := E^{n(j(k))}_{\{\sqrt{n(j(k)) - 1}\,b_{ik}\}_i}$.
It follows from Lemma \ref{lem:domE}
that $E_k$ is dominated by
$E^{n(j(k))}_{\{\sqrt{n(j(k)) - 1}\,a_{ij(k)}\}_i}$ for any $k$
and so Lemma \ref{lem:box-dom} implies that
$\{E_k\}$ has a box convergent subsequence.
However, Lemma \ref{lem:box-conv} proves that
$\{E_k\}$ has no box convergent subsequence,
which is a contradiction.
This completes the proof.
\end{proof}

%\begin{rem} \label{rem:exercise}
%In Exercise \cite{Gmv:green}*{3$\frac{1}{2}$.57(E)},
%the following four cases are presented for a sequence
%$\{\mathcal{E}^{n(j)}_{\{\sqrt{n(j)-1}\,a_{ij}\}_i}\}_j$ of solid ellipsoids:
%\begin{enumerate}
%\item $a_{ij} = \epsilon_i\sqrt{i/n(j)}$ and $\epsilon_i \to 0$ as $i\to\infty$.
%\item $a_{ij} = \sqrt{i/n(j)}$.
%\item $a_{ij} = a_i \to 0$ as $i\to\infty$.
%\item $a_{ij} = a_i$ and $\sum_{i=1}^\infty a_i^2 < +\infty$.
%\end{enumerate}
%
%We set $E_j := \mathcal{E}^{n(j)}_{\{\sqrt{n(j)}\,a_{ij}\}}$.
%
%In the case of (1),
%it is easy to see that $\sup_i a_{ij} \to 0$ as $j\to\infty$,
%so that $E_j$ converges to a one-point mm-space as $j\to\infty$
%by looking at the observable diameter of $E_j$.
%
%In the case of (2),
%we see that $\lim_{j\to\infty} a_{n(j)+1-i,j} = 1$ for each $i$,
%so that the sequence of the coordinate functions forms
% a non-precompact family in $(\Lip_1(E_j)/\sim,\dKF)$,
% which implies that $\{E_j\}$ does not asymptotically concentrates
%(see the proof of \cite{Sy:mmg}*{Proposition 7.37}).
%
%In the case of (3), \red{????}
%
%In the case of (4), the box convergence of $\{E_j\}$ follows from
%Proposition \ref{prop:box-conv}.
%
%\end{rem}

\end{document}